\documentclass[12pt]{amsart}
\usepackage{amssymb}
\usepackage{mathtools}
\usepackage[english]{babel}
\usepackage{epsfig}
\usepackage{mathptmx}
\usepackage{amsmath,amsfonts,amssymb}
\usepackage{mathtools}
\usepackage{mathrsfs}
\usepackage{graphicx}
\usepackage{latexsym}
\usepackage{verbatim}
\usepackage{enumerate}
\usepackage[svgnames]{xcolor}
\usepackage{ulem}
\usepackage{tabularx}
\usepackage{caption}
\usepackage[backref=page]{hyperref} 
\usepackage{wasysym}
\usepackage{relsize}
\usepackage{glossaries}
\usepackage{cancel}

\DeclareUnicodeCharacter{0080}{ }
\UseRawInputEncoding

\newcommand{\clos}[1]{\operatorname{clos}_{#1}}
\newcommand{\hide}[1]{}

\numberwithin{equation}{section}

\def\Re{{\sf Re}\,}

\def\eps{\varepsilon}

\newcommand{\D}{\mathbb D}

\newcommand{\R}{\mathbb R}

\newcommand{\Z}{\mathbb Z}
\newcommand{\C}{\mathbb C}

\newcommand{\Aut}{{\sf Aut}(\mathbb D)}

\newcommand{\N}{\mathbb N}

\def\Re{{\sf Re}\,}

\newcommand{\F}{{}_2F_1}
\newcommand{\Fz}[3]{{}_2F_1\left(#1;#2;#3\right)}

\def\Aut{{\sf Aut}}

\def\Re{{\sf Re}\,}

\def\Re{{\sf Re}\,}

\def\1#1{\overline{#1}}
\def\2#1{\widetilde{#1}}
\def\3#1{\widehat{#1}}
\def\4#1{\mathbb{#1}}
\def\5#1{\frak{#1}}
\def\6#1{{\mathcal{#1}}}

\def\Re{{\sf Re}\,}

\newcommand{\mcite}[1]{\csname b@#1\endcsname}

\newtheoremstyle{break}
{8pt}{8pt}%
{\itshape}{}%
{\bfseries}{}%
{\newline}{}%
\theoremstyle{break}

\def\Aut{{\sf Aut}}

\def\Re{{\sf Re}\,}

\newcommand{\cc}[1]{\overline{{#1}}}
\DeclarePairedDelimiter{\abs}{\lvert}{\rvert}

\newcommand{\del}{\mathop{}\!\partial}

\allowdisplaybreaks

\theoremstyle{break}

\newtheorem{theorem}{Theorem}[section]
\newtheorem*{theorem*}{Theorem}
\newtheorem{lemma}[theorem]{Lemma}
\newtheorem*{lemma*}{Lemma}
\newtheorem{proposition}[theorem]{Proposition}
\newtheorem{corollary}[theorem]{Corollary}

\theoremstyle{break}
\newtheorem{definition}[theorem]{Definition}
\newtheorem{example}[theorem]{Example}

\theoremstyle{remark}
\newtheorem{remark}{Remark}

\numberwithin{equation}{section}

\newmuskip\pFqmuskip

\newcommand*\pFq[6][8]{%
	\begingroup 
	\pFqmuskip=#1mu\relax
	\mathchardef\normalcomma=\mathcode`,
	\mathcode`\,=\string"8000
	\begingroup\lccode`\~=`\,
	\lowercase{\endgroup\let~}\pFqcomma
	{}_{#2}F_{#3}{\left[\genfrac..{0pt}{}{#4}{#5};#6\right]}%
	\endgroup
}
\newcommand{\pFqcomma}{{\normalcomma}\mskip\pFqmuskip}

\newcommand*\FF[3][8]{%
	\begingroup 
	\pFqmuskip=#1mu\relax
	\mathchardef\normalcomma=\mathcode`,
	\mathcode`\,=\string"8000
	\begingroup\lccode`\~=`\,
	\lowercase{\endgroup\let~}\pFqcomma
	{}_{3}F_{2}{\left[\genfrac..{0pt}{}{#2}{#3}\right]}%
	\endgroup
}

\definecolor{dgreen}{rgb}{0,0.5,0}

\newenvironment{mylist}{\begin{list}{}%
		{\labelwidth=2em\leftmargin=\labelwidth\itemsep=.4ex plus.1ex minus.1ex\topsep=.7ex plus.3ex minus.2ex}%
		\let\itm=\item\def\item[##1]{\itm[{\rm ##1}]}}{\end{list}}

\makeatletter
\def\blfootnote{\xdef\@thefnmark{}\@footnotetext}
\makeatother

\author[M. Heins]{Michael Heins}
\address{M. Heins: Department of Mathematics, University of W{\"u}rzburg, Emil Fischer Strasse 40, 97074, W{\"u}rzburg, Germany.} \email{michael.heins@mathematik.uni-wuerzburg.de}

\author[A.~Moucha]{Annika Moucha$^\dag$}
\address{A. Moucha: Department of Mathematics, University of W{\"u}rzburg, Emil Fischer Strasse 40, 97074, W{\"u}rzburg, Germany.} \email{annika.moucha@mathematik.uni-wuerzburg.de}

\author[O. Roth]{Oliver Roth}
\address{O. Roth: Department of Mathematics, University of W{\"u}rzburg, Emil Fischer Strasse 40, 97074, W{\"u}rzburg, Germany.} \email{roth@mathematik.uni-wuerzburg.de}

\title[Spectral theory of the invariant Laplacian]{Spectral theory of the invariant Laplacian\\[2mm] on the disk and the sphere – a complex analysis approach}
\thanks{$^\dag\,$Partially supported by the Alexander von Humboldt Stiftung}

\begin{document}


\begin{abstract}
The central theme of this paper is the holomorphic spectral theory of the canonical Laplace operator of the complement $\Omega \coloneqq \{(z,w) \in \widehat{\C}^2 \colon z \cdot w \neq 1\}$ of the ``complexified unit circle'' $\{(z,w) \in\widehat{\C}^2 \colon z \cdot w = 1\}$. We start by singling out a distinguished set of holomorphic eigenfunctions on the bidisk in terms of hypergeometric $\F$ functions and prove that they provide~a spectral decomposition of every holomorphic eigenfunction on the bidisk. As a second step, we identify the maximal domains of definition  of these eigenfunctions and show that these maximal domains naturally determine the fine structure of the eigenspaces. Our main result gives an intrinsic classification of all closed M{\"o}bius invariant subspaces of eigenspaces of the canonical Laplacian of $\Omega$. Generalizing foundational prior work of Helgason and Rudin, this provides~a unifying complex analytic framework for the real--analytic eigenvalue theories of both the hyperbolic and spherical Laplace operators on the open unit disk resp.~the Riemann sphere and, in particular, shows how they are interrelated with one another.
    \end{abstract}
\maketitle
	\section{Introduction}

Let $\widehat{\C}:=\C \cup \{\infty\}$ denote the Riemann sphere. The purpose of this paper is to explore the spectral theory of the complex \textit{invariant Laplace operator}
        \begin{equation*} \label{eq:InvLaplacian}
           \Delta_{zw}=4 (1-zw)^2 \partial_z \partial_w \,
         \end{equation*}
         of the  \textit{complement of the complexified unit circle},
        \begin{equation} \label{eq:Omega}
            \Omega:=\widehat{\C}^2 \setminus \left\{ (z,w) \in\widehat{\C}^2 \, : \, z\cdot w=1\right\}\footnotemark
        \end{equation}
by function--theoretic methods. \footnotetext{Here, we extend the arithmetic in $\C$ in the usual manner by $z \cdot \infty = \infty = \infty \cdot z$ for $z \in \widehat{\C} \setminus \{0\}$ and $0 \cdot \infty = 1 = \infty \cdot 0$.
        We think of the complexified unit circle as the set $\{(z,w) \in \widehat{\C}^2 \, : \, z \cdot w=1\}$.}
    This approach allows  a unified study of the {\textit{real--analytic}} spectral theories of the hyperbolic Laplacian $\Delta_\D:=4(1-|z|^2)^2 \partial_z\partial_{\overline{z}}$ on the open unit disk $\D:=\{z \in \C \, : \, |z|<1\}$ and the spherical Laplacian $\Delta_{\widehat{\C}}:=-4(1+|z|^2)^2 \partial_z \partial_{\overline{z}}$ on the Riemann sphere~$\widehat{\C}$ from a  complex analytic point of view and, in addition, it also shows how they are interrelated to one another. Beyond that, the complex point of view taken in this paper offers several other useful advantages. In particular, it connects in a natural way the fine structure of the eigenspaces of the hyperbolic and spherical Laplacians as described by Helgason \cite{Helgason70} and Rudin \cite{Rudin84} with the \textit{maximal} domain of existence of the corresponding holomorphic eigenfunctions of the invariant Laplacian~$\Delta_{zw}$.

    \medskip

    As one instance, we analyze from a  complex analysis point of view the building blocks of each $\lambda$--eigenspace of the hyperbolic Laplacian $\Delta_\D$, which have previously been identified e.g.~by Helgason in \cite{Helgason70} using real variable methods. It turns out that these so--called Poisson Fourier modes naturally extend to holomorphic eigenfunctions of $\Delta_{zw}$ which are (maximally) defined either on~$\Omega$ or on one of three distinguished subdomains of $\Omega$ depending on the choice of the eigenvalue~$\lambda$. This then allows a transparent proof that each holomorphic eigenfunction of $\Delta_{zw}$ defined on any rotationally invariant subdomain  of~$\Omega$ has a unique spectral decomposition in form of a  locally uniformly and absolutely convergent infinite series composed of Poisson Fourier modes. In the special case of the bidisk $\D^2$ and further restriction to the ``diagonal'' $\{(z,\overline{z}) \, : \, z \in \D\}$ we  recover the spectral decomposition of the smooth eigenfunctions of $\Delta_{\D}$ on the unit disk $\D$ as described e.g.~in \cite{Helgason70} or \cite{BerensteinGay1995}.

    \medskip

As a second instance, we investigate the structure of the closed ``M{\"o}bius invariant'' subspaces of any fixed eigenspace $X_\lambda(\D)$ of $\Delta_{\D}$ from a complex analysis point of view. This topic has been investigated in detail by Rudin in \cite{Rudin84} using purely ``real'' methods.
We shall see that the distinguished subdomains of $\Omega$ mentioned above naturally lead to the same distinction between exceptional and non--exceptional eigenvalues which have been found by Rudin. In Rudin's work, the exceptional cases correspond to the eigenvalues $\lambda=4m(m+1)$, $m=0,1,2\ldots$, and they are characterized by the existence of three non--trivial M{\"o}bius invariant closed subspaces of $X_{\lambda}(\D)$, exactly one of which, $X^0_\lambda(\D)$ say, is finite dimensional.  It turns out that a complex number $\lambda \in \C$ is an exceptional eigenvalue in the sense of Rudin if and only the invariant Laplacian $\Delta_{zw}$ has a globally (that is on $\Omega$) defined holomorphic $\lambda$--eigenfunction. Moreover, in this case the  unique finite dimensional invariant subspace  $X^0_{\lambda}(\D)$  corresponds precisely to the full $\lambda$--eigenspace of all globally defined $\lambda$--eigenfunctions of the invariant Laplacian $\Delta_{zw}$,  which then, in fact, is invariant under the full group of all M{\"o}bius transformations. For the other two non--trivial invariant subspaces of the exceptional $\Delta_\D$--eigenspace $X_\lambda(\D)$ discovered by Rudin as well as the full eigenspace $X_{\lambda}(\D)$ itself we give a similar but more intricate description in form of Runge--type approximation results in terms of holomorphic $\lambda$--eigenfunctions defined precisely on one of the distinguished three subdomains of $\Omega$, see Theorem \ref{thm:RudinIntro}.

\medskip

In the next section we give an account of the main results of this work and their ramifications for the spectral theory of the hyperbolic and spherical Laplacian as well as an outline of the structure of the remaining sections.
The accompanying papers \cite{HeinsMouchaRoth1,HeinsMouchaRoth2,KrausRothSchleissinger,Annika} are  related to other aspects of the function theory of the set $\Omega$, the complement of the complexified unit circle, and its applications. Our interest in this set and its inhabitants, the holomorphic functions on~$\Omega$, first arose in connection with previous work \cite{BeiserWaldmann2014,CahenGuttRawnsley1994,EspositoSchmittWaldmann2019,KrausRothSchoetzWaldmann, SchmittSchoetz2022,Waldmann2019} on  canonical Wick--type star products in  strict deformation quantization of the unit disk and the Riemann sphere, and from our desire to understand the somehow mysterious role played by $\Omega$ and in particular by its function--theoretic properties in this regard. A partial explanation was given in \cite{HeinsMouchaRoth2}, where it was indicated that invariant differential operators of Peschl--Minda type, on the one hand, effectively facilitate and unify the study of the star products on the disk and the sphere, and on the other hand, are perhaps best understood as operators acting on the spaces of holomorphic functions on~$\Omega$ and its three distinguished  subdomains.  We started wondering whether and how the most basic differential operator acting on $\Omega$, the invariant Laplacian $\Delta_{zw}$, and its spectral theory possibly fit into this emerging picture. This paper describes what we have found. In the forthcoming paper \cite{Annika} of the second--named author these endeavours will come to full circle:  it is shown that there are globally defined eigenfunctions of the Laplacian $\Delta_{zw}$ which
form a Schauder basis of the Fr\'echet space $\mathcal{H}(\Omega)$  of all holomorphic functions on $\Omega$. The results of the present paper then imply that the algebra   $$\mathcal{A}(\D)=\big\{f : \D \to \C \, \big| \, f(z)=F(z,\overline{z}) \text{ for all } z \in \D \text{ for some } F \in \mathcal{H}(\Omega)\big\} \, ,$$
for which the Wick--star product on $\D$ in \cite{KrausRothSchoetzWaldmann} is constructed, admits a spectral decomposition precisely into the finite dimensional invariant subspaces of the exceptional eigenspaces of the hyperbolic Laplace operator~$\Delta_{\D}$ on $\D$ discovered by Rudin \cite{Rudin84} many years ago. This provides an intrinsic characterization of the algebra $\mathcal{A}(\D)$ in terms of the natural hyperbolic geometry of the unit disk and its canonical invariant Laplacian~$\Delta_\D$.

    \section{Overview and main results} \label{sec:MainResults}

    In order to place the results of this paper into a broader context we begin by recalling in greater detail the striking distinction between \textit{exceptional} and \textit{non--exceptional} eigenvalues of the hyperbolic Laplace operator $\Delta_{\D}$ and its relevance for the
study of the \textit{invariant} subspaces of the $\Delta_{\D}$--eigenspaces which has been discovered by Rudin \cite{Rudin84}.
The Fr\'echet space of all twice continuously (real) differentiable functions $f : \D \to \C$ equipped with the standard compact--open topology is denoted by $C^2(\D)$.
For each $\lambda \in \C$ we denote by $X_{\lambda}(\D)$  the vector space of all $\lambda$--eigenfunctions $f \in C^2(\D)$ of the hyperbolic Laplacian, that is,
      $$ X_{\lambda}(\D)=\left\{ f \in C^2(\D) \, : \Delta_{\D} f=\lambda f \text{ on } \D \right\} \, . $$
It is known that each such $\Delta_{\D}$--eigenspace $X_{\lambda}(\D)$ is a closed infinite dimensional subspace of $C^2(\D)$.
      The hyperbolic Laplacian $\Delta_{\D}$ is invariant under the full group of all conformal automorphisms (biholomorphic maps) $T$ of the unit disk $\D$ in the sense that
      $$\Delta_{\D} (f \circ T)=\left(\Delta_{\D} f\right) \circ T \,\quad \text{ for all } f \in C^2(\D) \, .$$
      In order to emphasize that this group consists entirely of M{\"o}bius transformations, we call it  the \textit{M{\"o}bius group of $\D$} and denote it  by $\mathcal{M}(\D)$. A  closed subspace~$Y$ of $C^2(\D)$ is called \textit{M{\"o}bius invariant} if $f \circ \psi \in Y$ for all $f \in Y$ and all $\psi \in \mathcal{M}(\D)$.\footnote{M{\"o}bius invariant spaces are called $\mathcal{M}$--spaces in \cite{Rudin84}. We will reserve the symbol $\mathcal{M}$ for some other purpose.} It is called non--trivial if $Y\not=\{0\}$ and $Y\not=X$.

      \begin{theorem}[Rudin \cite{Rudin84}] \label{thm:IntroRudin}
        Let $\lambda \in \C$.
        \begin{itemize}
        \item[(NE)] If $\lambda\not=4m(m+1)$ for $m=0,1,2,\ldots$, then $X_{\lambda}(\D)$ has no non--trivial M{\"o}bius invariant subspaces.
        \item[(E)] If $\lambda=4m(m+1)$ for some $m=0,1,2,\ldots$, then  $X_{\lambda}(\D)$ has precisely three distinct non--trivial M{\"o}bius invariant subspaces. There is exactly one non--trivial M{\"o}bius invariant subspace of $X_{\lambda}(\D)$ which is finite dimensional; its dimension is $2m+1$.
             \end{itemize}
        \end{theorem}

        The alternative (E) in Theorem \ref{thm:IntroRudin} will be called the \textit{exceptional} case and the unique finite dimensional M{\"o}bius invariant subspace of $X_{\lambda}(\D)$ will be denoted by $X^0_\lambda(\D)$. The alternative (NE) will be referred to as the \textit{non--exceptional} case.

        \medskip

        One of the main results of the present paper is a complete analogue of Theorem \ref{thm:IntroRudin} with $\Delta_{\D}$ replaced by the differential operator $\Delta_{zw}$, see Theorem \ref{thm:2Intro}. Apart from being potentially interesting in its own right, it  provides a concrete function--theoretic description of the exceptional eigenspaces in Theorem \ref{thm:IntroRudin}. This also adds a conceptual component to Rudin's handling of the invariant eigenspaces of the hyperbolic Laplace operator $\Delta_\D$.

         \medskip

         The characteristic feature of our approach is  to  look for \textit{holomorphic} solutions $F$ of the eigenvalue equation
        \begin{equation} \label{eq:EWE}
            \Delta_{zw} F=\lambda F \,
        \end{equation}
        defined on a subdomain $D$ of $\Omega$ which we wish to choose as large as possible depending on the eigenvalue $\lambda$. These \textit{maximal domains of existence} (see Definition~\ref{def:MaximalDomain}) of the $\lambda$--eigenfunctions turn out to be the only essential ingredients which are needed to give a complete description of the (invariant) $\lambda$--eigenspaces of the operator $\Delta_{zw}$ and its offsprings $\Delta_{\D}$ and $\Delta_{\widehat{\C}}$.

        \medskip

        In order to state our main results we have to adapt the notation which we have introduced above for the hyperbolic Laplacian to the case of the differential operator $\Delta_{zw}=4(1-zw)^2 \partial_z\partial_w$.
        Instead of working in the Fr\'echet space $C^2(\D)$ we now fix a subdomain $D$ of the set $\Omega:=\{(z,w) \in \widehat{\C} \, : \, z \cdot w\not=1\}$, and  work in the Fr\'echet space $\mathcal{H}(D)$ of all complex--valued holomorphic functions defined on $D$ (again equipped with the topology of locally uniform convergence, this time on~$D$).
        Our goal is determine the holomorphic solutions $F : D \to \C$ of  the eigenvalue equation (\ref{eq:EWE}),
 i.e. we are interested in the $\Delta_{zw}$--eigenspaces
        $$ X_{\lambda}(D):=\left\{ F \in \mathcal{H}(D) \, : \, \Delta_{zw} F=\lambda F \text{ on } D \right\} \, , \qquad \lambda \in \C \, .$$
        With regard to Rudin's theorem (Theorem \ref{thm:IntroRudin}) a particularly  natural choice for the domain $D$ is the \textit{bidisk} $\D^2:=\D \times \D$ since it is easy to see that for every $F \in X_{\lambda}(\D^2)$ the ``restriction'' of~$F$ to the ``diagonal'' $\{(z,\overline{z}) \, : \, z \in \D\}$, that is, $f(z):=F(z,\overline{z})$, $z \in \D$, yields an eigenfunction $f \in C^2(\D)$ of the hyperbolic Laplacian $\Delta_{\D}$ for the eigenvalue $\lambda$.
       In fact, the following result shows that every $f \in X_{\lambda}(\D)$ arises in this fashion, that is,  it has an ``extension'' to an  eigenfunction $F$ of~$\Delta_{zw}$ which is holomorphic on the bidisk $\D^2$:

		\begin{theorem}[Smooth eigenfunctions of $\Delta_{\D}$ on $\D$ vs.~holomorphic eigenfunctions of $\Delta_{zw}$ on~$\D^2$] \label{thm:DiskvsBidiskIntro}
                  Let $\lambda \in \C$ and $f \in C^{2}(\D)$ such that $\Delta_{\D} f=\lambda f$ in $\D$. Then there is a uniquely determined function $F \in \mathcal{H}(\D^2)$ such that $\Delta_{zw}F=\lambda F$ on $\D^2$ and $f(z)=F(z,\overline{z})$ for all $z \in \D$.  Moreover, the induced bijective linear map
                  $$ X_{\lambda}(\D) \to X_{\lambda}(\D^2)$$
                  is continuous.
		\end{theorem}

                \begin{remark}
                  If $f \in C^2(\D)$ solves $\Delta_{\D} f=\lambda f$ in $\D$, then $f$ is real--analytic (\cite[Theorem 4.2.5]{Rudin2008}, so there is trivially a holomorphic function $\tilde{F}$ defined on \textit{some} open neighborhood $U \subseteq \D^2$ of the diagonal $\{(z,\overline{z}) \in \D^2 \, : \, z \in \D$\} such that $\tilde{F}(z,\overline{z})=f(z)$ for all $z \in \D$ satisfying $(z,\overline{z}) \in U$. In view of a  well--known variant of the  identity principle (see Lemma \ref{lem:RangeIDprincpiple}) there is only one such holomorphic extension $\tilde{F} : U \to \C$ of $f \in C^2(\D)$ to $U$. The point of Theorem \ref{thm:DiskvsBidiskIntro} is that $\tilde{F}$ has a holomorphic extension (at least) to the bidisk~$\D^2$.
                  \end{remark}

                Theorem \ref{thm:DiskvsBidiskIntro}  gives rise to the following definition.

                \begin{definition}
                  Let $\lambda \in \D$. Then the continuous bijective linear mapping   $$ \mathcal{R}_h : X_{\lambda}(\D^2) \to X_{\lambda}(\D) \, , \qquad \mathcal{R}_h(F)(z):=F(z,\overline{z}) \quad (z \in \D) \, $$
is called the \textbf{hyperbolic restriction map}. Its continuous  inverse
$$ \mathcal{E}_h :={\left(\mathcal{R}_{h}\right)}^{-1} : X_{\lambda}(\D) \to X_{\lambda}(\D^2) \,  $$
is called the \textbf{hyperbolic extension map} from $X_{\lambda}(\D)$ to $X_{\lambda}(\D^2)$.
\end{definition}

The hyperbolic restriction and extension mappings provide the bridge between the holomorphic spectral theory of the invariant Laplacian $\Delta_{zw}$ and the smooth spectral theory of the hyperbolic Laplacian $\Delta_{\D}$.
In particular, one can study the spectral properties of $\Delta_{\D}$ on $\D$ from the viewpoint of complex analysis on the bidisk $\D^2$.
                Based on Theorem \ref{thm:DiskvsBidiskIntro}, we can now proceed to associate the fine structure of the eigenspaces $X_{\lambda}(\D)$ with the \textit{maximal domains of existence} of holomorphic eigenfunctions in $X_{\lambda}(\D^2)$, a concept which is defined as follows, cf.~\cite[p.~97]{FritzscheGrauert}.

    \begin{definition}
        \label{def:MaximalDomain}
        Let $F$ be a holomorphic function on the bidisk $\D^2$.
        A subdomain $D \subseteq \Omega$ that contains the bidisk~$\D^2$ is called a \textbf{maximal domain of existence of $F$} if the function $F$ has a holomorphic extension to $D$ but to no strictly larger subdomain of $\Omega$.
    \end{definition}

                  For our purposes, this definition is natural in several respects. First, the condition $D \supseteq \D^2$ obviously comes from Theorem \ref{thm:DiskvsBidiskIntro}. Second, the condition $D \subseteq \Omega$ is natural in view of the fact that every function which is holomorphic on a  subdomain of $\widehat{\C}^2$  that is strictly larger than~$\Omega$ is necessarily constant (see Theorem 5.3 in \cite{HeinsMouchaRoth1}). In particular, the largest possible maximal domain of existence of any  eigenfunction of $\Delta_{zw}$ in $\mathcal{H}(\D^2)$ is $\Omega$. Notably, this includes the constant eigenfunctions.  At first sight,  such a linguistic subtlety might seem irritating, but it facilitates the statement of many of our results.

                  \medskip

With this concept at hand, we can now  give a function--theoretic characterization of the exceptional eigenvalues of the hyperbolic Laplacian $\Delta_{\D}$ and the corresponding unique finite--dimensional non--trivial M{\"o}bius invariant subspaces $X_\lambda^0(\D)$ of the $\Delta_{\D}$--eigenspaces $X_{\lambda}(\D)$. In addition, the following theorem provides an equivalent condition in terms of existence of globally defined eigenfunctions of the spherical Laplacian $\Delta_{\widehat{\C}}$.
	\begin{theorem}[Exceptional smooth eigenfunctions $\Delta_{\D}$ on $\D$ vs.~holomorphic eigenfunctions of $\Delta_{zw}$ on $\Omega$ vs.~smooth eigenfunctions of $\Delta_{\widehat{\C}}$ on $\widehat{\C}$] \label{thm:2Intro}
    Let $\lambda \in \C$. Then the following conditions are pairwise equivalent:
  \begin{itemize}
    \item[(i)]    \label{item:ExceptionalOmegaExists}
    There exists a function $F \in X_{\lambda}(\D^2)$ with $\Omega$ as maximal domain of existence.
  \item[(ii)] \label{item:ExceptionalSphereExists}
  There exists a function $g \in C^{2}(\widehat{\C})$ such that $\Delta_{\widehat{\C}} g=\lambda g$ on $\widehat{\C}$.
  \item[(iii)] \label{item:ExceptionalEigenvalue}
  $\lambda$ is an exceptional eigenvalue of $\Delta_{\D}$, i.e.\ $\lambda=4 m (m+1)$ for some non--negative integer~$m$.
  \end{itemize}

  If one of these conditions is in place  and $\lambda=4m(m+1)$, then $\dim X_{\lambda}(\Omega)=2m+1$ and the following statements hold:
  \begin{itemize}
    \item[(a)] (Function--theoretic description of  $X_\lambda^0(\D)$)

   The domain  $\Omega$ is the maximal domain of existence for every $F \in X_\lambda^0(\D^2)$, i.e.\
   $$ X^0_{\lambda}(\D)=\mathcal{R}_h\left(X_{\lambda}(\Omega)\right) \, .$$
\item[(b)] (Function--theoretic description of the smooth eigenfunctions of $\Delta_{\widehat{\C}}$ on $\widehat{\C}$)
\begin{equation}
    \label{eq:SphericalEigenfunctions}
    \left\{ g \in C^2(\widehat{\C}) \, : \, \Delta_{\widehat{\C}} g=\lambda g \text{ on } \widehat{\C}\right\}=\big\{\widehat{\C} \ni z \mapsto F(z,-\overline{z}) \, : \, F \in X_{\lambda}(\Omega) \big\} \, .
\end{equation}
     \end{itemize}
\end{theorem}

\begin{remark}
  \begin{itemize}
\item[(a)] Note that $X_0^0(\D^2)$ and $X_0(\Omega)$ consist precisely of the constant functions.
    \item[(b)]
      In view of Theorem \ref{thm:2Intro}, an eigenvalue $\lambda$ of $\Delta_{\D}$ is exceptional if and only if there exists a holomorphic eigenfunction of $\Delta_{zw}$ on $\D^2$ with the largest possible maximal domain of existence, the set $\Omega$.
    \item[(c)] In Theorem \ref{thm:2Intro} (a)  we think of $X_{\lambda}(\Omega)$ as a subspace of $X_{\lambda}(\D^2)$. It is a closed subspace of $X_{\lambda}(\D^2)$ (in the topology of $\mathcal{H}(\D^2)$) because it is finite dimensional.
    \item[(d)] (Spherical restriction map)\\
      If $\lambda \in \C$ is an exceptional eigenvalue, then the \textit{spherical restriction map}
      $$ \mathcal{R}_s : X_{\lambda}(\Omega) \to C^2(\widehat{\C}) \, , \qquad \mathcal{R}_s(F)(z):=F(z,-\overline{z}) \qquad (z \in \widehat{\C}) $$
      is well--defined. By Theorem \ref{thm:2Intro} (b), $\mathcal{R}_s$ is a bijection from $X_{\lambda}(\Omega)$ onto the $\lambda$--eigenspace
      $\big\{ g \in C^2(\widehat{\C}) \, : \, \Delta_{\widehat{\C}} g=\lambda g \text{ on } \widehat{\C}\big\}$  of $\Delta_{\widehat{\C}}$, which is clearly continuous.
\end{itemize}
\end{remark}

With Theorem \ref{thm:2Intro} we have reached two of our goals, a conceptual characterization of exceptional eigenvalues and the finite dimensional non--trivial M{\"o}bius invariant subspaces of the eigenspaces $X_{\lambda}(\D)$ of the hyperbolic Laplacian $\Delta_\D$. Next, we address the infinite dimensional non--trivial M{\"o}bius invariant subspaces of the $\Delta_{\D}$--eigenspaces $X_{\lambda}(\D)$. This turns out to be more difficult, and first requires clarification of the invariance properties of the Laplacian $\Delta_{zw}$. Implicitly, the underlying difficulty is already present in Theorem \ref{thm:2Intro}, and can be seen as follows. Let $\lambda \in \C$ be an exceptional eigenvalue of $\Delta_{\D}$. Then the finite--dimensional  M{\"o}bius invariant subspace $X^0_\lambda(\D)$ is invariant under all automorphisms of~$\D$. However, the corresponding eigenspace $X_{\lambda}(\Omega)$ for $\Delta_{zw}$, which consists of functions holomorphic on $\Omega$, is, loosely speaking, invariant under a much larger group of automorphisms. In fact, we  first note that the invariant Laplacian~$\Delta_{zw}$ of~$\Omega$ is \textit{not} invariant under all biholomorphic automorphisms of $\Omega$ in the sense that the invariance condition
\begin{equation} \label{eq:InvOmega}
  \Delta_{zw} \left( F \circ T \right)=\left(\Delta_{zw} F\right) \circ T \qquad \text{ for all } F \in \mathcal{H}(\Omega) \,
\end{equation}
holds  for all $T \in \Aut(\Omega)$.
The reason is simply that the automorphism group $\Aut(\Omega)$ is much too large, see \cite{HeinsMouchaRoth1}. However, $\Delta_{zw}$ \textit{is} invariant under the subgroup $\mathcal{M}$ of $\Aut(\Omega)$ defined by
 \begin{equation}\label{eq:OmegaAutomorphismIntro}
    \mathcal{M}
    :=
    \bigcup \limits_{ \psi \in \mathcal{M}(\widehat{\C})}
    \left\{ (z,w) \mapsto
        \left(
    \psi(z),\frac{1}{\psi(1/w)}
    \right)\, , \,
        (z,w) \mapsto
        \left(
    \psi(w),\frac{1}{\psi(1/z)}
    \right) \right\}\, ,
\end{equation}
where we write $\mathcal{M}(\widehat{\C})$ for the group of all M{\"o}bius transformations $\psi:\widehat{\C}\to\widehat{\C}$.
This is easy to prove by direct verification. (Conversely, one can show that every $T \in \Aut(\Omega)$ for  which the invariance property (\ref{eq:InvOmega}) holds does necessarily belong to the subgroup $\mathcal{M}$, see \cite[Theorem 5.2]{HeinsMouchaRoth1}, but we do not need such a result in this paper.) Since $\mathcal{M}$ is ``induced'' by  the set $\mathcal{M}(\widehat{\C})$ of all M{\"o}bius transformations, we call~$\mathcal{M}$ the \textit{M{\"o}bius group of $\Omega$}. Note that $\mathcal{M}(\widehat{\C})$ is strictly bigger than $\mathcal{M}(\D)$\footnote{$\mathcal{M}(\widehat{\C})$ is real six dimensional, while $\mathcal{M}(\D)$ is real three dimensional.}, so the M{\"o}bius group~$\mathcal{M}$ of $\Omega$ is strictly larger than the M{\"o}bius group $\mathcal{M}(\D)$ of $\D$. Now, while $X^0_\lambda(\D)$ is invariant under each element of $\mathcal{M}(\D)$, the set $X_{\lambda}(\Omega)$ is invariant under each element of $\mathcal{M}$, simply because $X_{\lambda}(\Omega)$ is the entire $\lambda$--eigenspace of the Laplacian $\Delta_{zw}$ on $\Omega$ and $\Delta_{zw}$ is invariant with respect to the M{\"o}bius group~$\mathcal{M}$.

\medskip

In view of this discussion, it is now clear that a suitable concept of invariance for the eigenspaces $X_{\lambda}(\D^2)$ of $\Delta_{zw}$ on the bidisk $\D^2$ has to be based on the group
$$\mathcal{M}(\D^2):=\Aut(\D^2) \cap \mathcal{M} \, ,$$
which we call the \textit{M{\"o}bius group of the bidisk $\D^2$}. It consists precisely of all automorphisms of the bidisk $\D^2$ which have the invariance property (\ref{eq:InvOmega}).
In fact, it is not difficult to show that
\begin{equation}
    \label{eq:OmegaAutomorphismBidisk}
    \mathcal{M}(\D^2) = \bigcup \limits_{ \psi \in \mathcal{M}(\D)}
    \left\{ (z,w) \mapsto
        \left(
    \psi(z),\frac{1}{\psi(1/w)}
    \right)\, , \,
        (z,w) \mapsto
        \left(
    \psi(w),\frac{1}{\psi(1/z)}
  \right) \right\}\,  .
\end{equation}

Clearly, each $\Delta_{zw}$--eigenspace $X_{\lambda}(\D^2)$ is invariant with respect to $\mathcal{M}(\D^2)$, that is, whenever ${F \in X_{\lambda}(\D^2)}$ and $T \in \mathcal{M}(\D^2)$, then $F \circ T \in X_{\lambda}(\D^2)$.
Each closed subspace  of $\mathcal{H}(\D^2)$  which is invariant under the M{\"o}bius group  $\mathcal{M}(\D^2)$  will be called a \textit{M{\"o}bius invariant subspace}. Theorem \ref{thm:DiskvsBidiskIntro} implies immediately the following result.

\begin{corollary}\label{cor:Intro1} Let $\lambda \in \C$ and $Y$ a subspace of $X_{\lambda}(\D)$. Then $Y$ is M{\"o}bius invariant if and only if    $\mathcal{E}_h(Y)$ is a  M{\"o}bius invariant subspace of $X_{\lambda}(\D^2)$.
  \end{corollary}

  We can now give a function--theoretic characterization of the M{\"o}bius invariant subspaces of $X_{\lambda}(\D^2)$ and thereby, in view of
  Corollary \ref{cor:Intro1}, the M{\"o}bius invariant subspaces of $X_{\lambda}(\D)$.
  The following subdomains of $\Omega$ play the essential role for this purpose:
  			\begin{align*}
                \Omega_+
				&\coloneqq
				\Omega
				\setminus
				\big\{
				(z,\infty)
				\colon
				z \in \hat{\C}
				\big\} \, .\\
                \Omega_-
				&\coloneqq
				\Omega
				\setminus
				\big\{
				(\infty,w)
				\colon
				w \in \hat{\C}
				\big\}\, .\\
                \Omega_*\, & := \left\{(z,w) \in \C^2\,:\, zw \in \C \setminus [1,\infty)\right\} \, .
			\end{align*}
                Note that each of these three subdomains of $\Omega$ contains the bidisk $\D^2$. Moreover,
                $$ \Omega=\Omega_+ \cup \Omega_- \cup \{ (\infty,\infty)\}$$ and
                        $$\mathcal{H}(\Omega)=\mathcal{H}(\Omega_+) \cap \mathcal{H}(\Omega_-) \, .$$
                         Figure~\ref{fig:OmegaPM} provides a schematic view of $\Omega$ and its distinguished subsets.

	\begin{figure}[h]
		\begin{minipage}{4\textwidth/18}
			\centering
			\includegraphics[width = \textwidth]{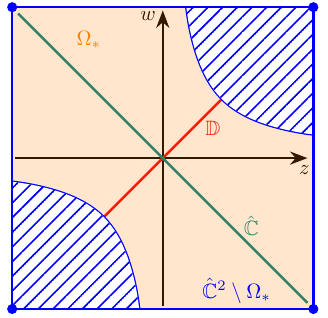}
		\end{minipage}
		\begin{minipage}{4\textwidth/18}
			\centering
			\includegraphics[width = \textwidth]{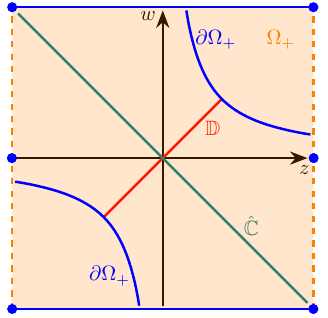}
		\end{minipage}
		\begin{minipage}{4\textwidth/18}
			\centering
			\includegraphics[width = \textwidth]{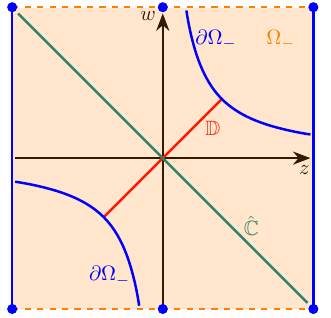}
		\end{minipage}
		\begin{minipage}{4\textwidth/18}
			\centering
			\includegraphics[width = \textwidth]{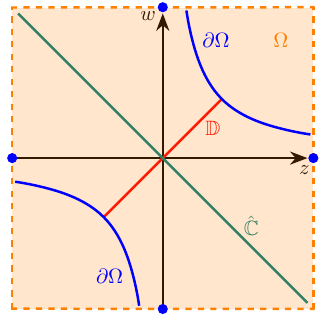}
		\end{minipage}
		\caption{Schematic picture of the sets $\Omega_*$, $\Omega_+$, $\Omega_-$ and $\Omega$ (from left to right) with points at infinity. Here, $\D$ is identified with the diagonal $\{(z,\overline{z}) \, : \, z \in \D\}$ and $\widehat{\C}$ with the rotated diagonal $\{(z,-\overline{z}) \, : \, z \in \widehat{\C}\}$.}
        \label{fig:OmegaPM}
              \end{figure}

We note in passing that the subdomains $\Omega_+$ and $\Omega_-$ arise naturally in the study of the Fr\'echet space structure of $\mathcal{H}(\Omega)$ (see \cite{HeinsMouchaRoth1}) and also for studying invariant differential operators of Peschl--Minda type acting on $\mathcal{H}(\Omega)$ (see \cite{HeinsMouchaRoth2}). The following result shows that they are also useful for describing the M{\"o}bius invariant subspaces of the eigenspaces of  the invariant Laplacian $\Delta_{zw}$. We use the following terminology:
		\begin{definition}
			Let $U \subseteq V$ be subdomains of some complex manifold, and
			let $Y \subseteq X$ be subsets of $\mathcal{H}(U)$. We say that
			\textbf{``$X \cap \mathcal{H}(V)$ is dense in $Y$''} if $X \cap \mathcal{H}(V) \subseteq Y$ and  if every function in $Y$ can be approximated locally uniformly on $U$ by functions in $X$, which have a holomorphic extension to~$V$.
		\end{definition}

                We are now, finally, in a position to formulate the main result of this paper.

		\begin{theorem}
			\label{thm:RudinIntro}
			Let $\lambda \in \C$ and let $Y$ be a non--trivial M{\"o}bius invariant subspace of
                        $X_{\lambda}(\D^2)$.
			Then one and only one of the following four alternatives holds.
			\begin{itemize}
				\item[($\text{E}_0$\,)]
				\label{item:E0}
				$Y=X_{\lambda}(\D^2) \cap \mathcal{H}(\Omega)$ and $\dim Y<\infty$.
				\item[($\text{E}_+$)]
				\label{item:E+}
				$X_{\lambda}(\D^2) \cap \mathcal{H}(\Omega_+)$ is dense in $Y$.
				\item[($\text{E}_-$)]
				\label{item:E-}
				$X_{\lambda}(\D^2) \cap \mathcal{H}(\Omega_-)$ is dense in $Y$.
				\item[(NE)]
				\label{item:NE}
				$X_{\lambda}(\D^2) \cap \mathcal{H}(\Omega_*)$ is dense in $Y$. In this case $Y=X_{\lambda}(\D^2)$.
			\end{itemize}
			In addition,
			\begin{itemize}
				\item[(i)] {\rm ($\textrm{E}_0$\,)}, {\rm ($\textrm{E}_+$\,)} and {\rm ($\textrm{E}_-$\,)} cannot occur if $\lambda\not=4m(m+1)$ for every non--negative integer $m$.
				\item[(ii)] {\rm ($\textrm{E}_0$\,)} holds if and only if $\dim Y<\infty$. In this case, $\dim Y=2m+1$ and $\lambda=4m(m+1)$ for some non--negative integer $m$.
				\item[(iii)] None of the density statements {\rm ($\textrm{NE}$\,)}, {\rm ($\textrm{E}_+$\,)} and {\rm ($\textrm{E}_-$\,)} may be improved to equalities.
			\end{itemize}
                      \end{theorem}

                  		\begin{remark}
			Recall that two domains $U \subseteq V$ in $\C^n$ form a Runge pair $(U,V)$ if every function in $\mathcal{H}(U)$ can be approximated locally uniformly in $U$ by functions in $\mathcal{H}(V)$. Identifying domains $U\subseteq V$ as Runge pairs is a fundamental and in many cases challenging problem in complex analysis. Note that
			in our terminology $(U,V)$ is a Runge  pair if and only if  $\mathcal{H}(U) \cap \mathcal{H}(V)$ is dense in $\mathcal{H}(U)$. For  subspaces $Y \subseteq X$ of $\mathcal{H}(U)$
			it is tempting to call $(Y,X,\mathcal{H}(V))$ a \textit{Runge triple}, if $X \cap \mathcal{H}(V)$ is dense in $Y$.  Then for every non--trivial infinite dimensional  M{\"o}bius invariant subspace $Y$ of $X_{\lambda}(\D^2)$ exactly one of the triples
 $$\left(Y,X_{\lambda}(\D^2),\mathcal{H}(\Omega_+)\right) \, ,\qquad  \left(Y,X_{\lambda}(\D^2),\mathcal{H}(\Omega_-)\right) \, ,\qquad  \left(Y,X_{\lambda}(\D^2),\mathcal{H}(\Omega_*)\right) $$
                          is a Runge triple.			Theorem \ref{thm:RudinIntro} can therefore be regarded as a Runge--type approximation theorem for the M{\"o}bius invariant spaces of eigenfunctions of the invariant Laplacian $\Delta_{zw}$.
                      \end{remark}

                      The plan of the paper is as follows. We introduce some basic concepts and notation in a preliminary Section \ref{sec:prelim}.
                      In Section \ref{sec:RudinsODE} and Section \ref{sec:Helgason} we develop the general spectral theory of the invariant Laplacian $\Delta_{zw}$ on $\Omega$ in analogy to the well--established spectral theory of the hyperbolic Laplacian $\Delta_\D$ on $\D$. In a sense, we  rather closely follow the presentation Berenstein and Gay \cite[Section 1.6]{BerensteinGay1995} have given for the spectral theory of the hyperbolic Laplacian, but we have made an effort either to provide even more rigorous proofs or to give precise references to the literature for all  auxiliary results which are needed. In contrast to \cite{BerensteinGay1995} we completely work in the holomorphic setting. On the one hand, this makes it possible to take advantage of many efficient tools from complex analysis which are not available otherwise. On the other hand, we need to incorporate from the beginning  the maximal domain of existence of eigenfunctions; an issue which does not even show up when working ``only'' on the unit disk. Here, our approach requires some finer analysis of the building blocks of the eigenfunctions, namely certain  hypergeometric functions and their integral representations in terms of Poisson Fourier modes.

        In Section \ref{sec:CompareRudin Helgson} we prove Theorem \ref{thm:DiskvsBidiskIntro} and show that it is in some sense best possible by providing an explicit example. This implies that the smooth spectral theory of the hyperbolic Laplacian $\Delta_{\D}$ on the disk $\D$ and the holomorphic spectral theory of the Laplacian $\Delta_{zw}$ on the bidisk $\D^2$ are essentially equivalent. In the same spirit we relate the smooth spectral theory of the spherical Laplacian $\Delta_{\widehat{\C}}$ with the holomorphic spectral theory of the Laplacian $\Delta_{zw}$ on $\Omega$ as well as the exceptional eigenvalues of the hyperbolic Laplacian $\Delta_{\D}$, see Theorem \ref{thm:SpherevsBiPlane}.
        Section \ref{sec:Pullback} is devoted to a study of the transformation behavior of the Poisson Fourier modes under precompositions with elements of the M{\"o}bius group $\mathcal{M}$. These results are needed for Section \ref{sec:PFandRudin} where we prove our main results,  Theorem \ref{thm:2Intro} and Theorem \ref{thm:RudinIntro}. By and large, we  follow Rudin's \cite{Rudin84} treatment of invariant subspaces of eigenfunctions of the hyperbolic Laplacian, but again completely working in the holomorphic setting; we briefly comment on the similarities and differences between our and Rudin's approach in Remark \ref{rem:RudintheoremOmega}. We close the paper with Section \ref{sec:PoissonFouriervsPeschlMinda} which connects the Poisson Fourier modes to the invariant differential operators of Peschl--Minda type studied in  \cite{HeinsMouchaRoth2}.

\medskip

Four final preliminary remarks are in order. First, treating the hyperbolic eigenvalue equation $\Delta_{\D}f=\lambda f$ and the spherical eigenvalue equation $\Delta_{\widehat{\C}} f=\lambda f$ as special cases of the more general complex eigenvalue equation $\Delta_{zw} F=\lambda F$ has been a recurrent theme in the literature for a long time.  To mention but a few of the many references, we refer for instance to the papers \cite{BauerI, BauerII, BauerRuscheweyh} and their bibliographies.  What seems to be new is the systematic study of the maximal domains of existence of the holomorphic solutions of $\Delta_{zw} F=\lambda F$ and its ramifications for the study of the invariant subspaces of the $\Delta_{\D}$--eigenspaces.
As a second remark, we should point out that Rudin's work \cite{Rudin84} is in fact concerned with the invariant Laplace operator on the unit ball of~$\C^n$, while our focus is exclusively on the complex one--dimensional case $n=1$. Thirdly, even though we are superficially dealing with holomorphic functions of two complex variables, we only need very few and only elementary facts from the theory of several complex variables.
Finally, Helgason \cite{Helgason1959,Helgason70,Helgason84} has systematically studied invariant differential operators and their eigenvalue problem in the setting of homogeneous spaces. In contrast to our holomorphic approach, he used entirely real methods. This Lie theoretic approach has since been generalized significantly. While providing a comprehensive list of references would go beyond what we can achieve here, we would like to mention \cite{Kashiwara,Koryani}, who generalized the theory to higher dimensional symmetric spaces, and Maa{\ss} \cite{Maass1949}, who initiated the vast and fruitful research of Maa{\ss} wave forms.

 \section{Notation and Preliminaries} \label{sec:prelim}

We denote the open unit disk in $\C$ by $\D$, the bidisk $\D^2:=\D\times\D$ and the Riemann sphere by~$\widehat{\C}$. The open disk of radius $r>0$ centered at the origin is denoted by $\D_r$.
   Moreover, we write $\N \coloneqq \{1, 2, \ldots\}$ for the set of positive integers, $\N_0 \coloneqq \N \cup \{0\}$ and $\Z$ for the set of all integers.
  For an open subset $U$ of $\widehat{\C}$ or $\widehat{\C}^2$, we write $\partial U$ for its boundary and $\cc{U}$ for its closure. The set of all twice resp.\ infinitely (real) differentiable functions $f:U\to\C$ is denoted $C^2(U)$ resp.\ $C^\infty(U)$, and we write $\mathcal{H}(U)$ for set of all holomorphic functions $f:U\to\C$.\\
  The set $\Omega=\widehat{\C}^2 \setminus \{(z,w) \in \widehat{\C}^2 \, : \, z \cdot w \not=1\}$  is a complex manifold of complex dimension $2$ and an open  submanifold of $\widehat{\C}^2$. In order to describe the complex structure of $\Omega$ only two charts are necessary, the \textit{standard chart}
  $$ \Omega \cap (\C \times \C) \to \C^2 \, , \qquad  (u,v) \mapsto (u,v)$$
  and the \textit{flip chart}
  $$ \Omega \cap \left( \widehat{\C} \setminus \{0\} \times \widehat{\C} \setminus \{0\} \right) \to \C^2 \, , \qquad (u,v) \mapsto (1/v,1/u) \, .$$
  In these local coordinates the \textit{invariant Laplace operator} $\Delta_{zw}$ of the complex manifold $\Omega$ is then  given by
  $$ \Delta_{zw} f(z,w)=4 (1-zw)^2 \partial_z \partial_w f(z,w) \, ,$$
  and it is easily seen that $\Delta_{zw}$  is a well--defined object. Here, $\partial_z$ and $\partial_w$ denote the Wirtinger derivatives with respect to $z$ and $w$, respectively.\\

{One of the few elementary results from the theory of functions of several complex variables we need in this paper is the following simple lemma.
 \begin{lemma}[``Two variable identity principle'' (p.~18 in \cite{Range})]\label{lem:RangeIDprincpiple}
Let $U$ be a subdomain of $\C^2$ which contains a point of the form $(z,\cc{z})$ resp.~$(z,-\cc{z})$. Let $f : U \to \C$ be a holomorphic function such that $$\{f(z,\overline{z}) \, : \, (z,\overline{z}) \in U\}=\{0\}\quad\text{resp.}\quad \{f(z,-\overline{z}) \, : \, (z,-\overline{z}) \in U\}=\{0\}\,.$$ Then $f \equiv 0$.
 \end{lemma}

Finally, we briefly recall some standard terminology from linear algebra. We denote by $\mathrm{span}\,\,M$ the collection of all \textit{finite} linear combinations of elements of a subset $M$ of~a given vector space~$X$. The vector spaces that occur in this paper are spaces of smooth or holomorphic functions defined on some open subset $U$ of $\Omega$ which we equip with the standard topology of uniform convergence on compact subsets of $U$.
If $M$ denotes a set of smooth or holomorphic functions on $U$, we denote by $\clos{U}M$ the closure of $M$ with respect to locally uniform convergence on $U$.

	\section{Homogeneous eigenfunctions and Poisson  Fourier modes}\label{sec:RudinsODE}\

By making a separation of variables approach Rudin \cite{Rudin84} showed that every $\lambda\in\C$ is an eigenvalue of $\Delta_\D$. The analogous result is true for $\Delta_{zw}$: let $n \in \Z$ and suppose $f_n : D \to \C$ is a holomorphic function defined on a domain $D \subseteq \C^2$ containing the origin $(0,0)$. Further assume that $f_n$ is \textit{$n$--homogeneous}, i.e.,
\begin{equation} \label{eq:3.1}
f_n(\eta z,w/\eta)=\eta^n f_n(z,w) \qquad \text{for all } \eta \in \partial \D
\end{equation}
and for all $(z,w) \in \C^2$ belonging to the bidisk $\D_r^2=\D_r \times \D_r$ for some (and hence all) $r >0$ such that $\D_{r}^2\subseteq D$. A consideration of the power series expansion of $f_n$ at $(0,0)$ in $\D_r^2$ implies in a straightforward way that there is a holomorphic function $y_n : \D_{r^2} \to \C$ such that
\begin{equation} \label{eq:3.1New}
 f_n(z,w)=y_n(zw) z^n \quad \text{ if } n \ge 0 \quad \text{ and } \quad f_n(z,w)=y_n(zw) w^{|n|} \quad \text{ if } n \le 0 \,
 \end{equation}
 for all $(z,w) \in \D_r^2$.
It is, then, easy to see that  $\Delta_{zw} f_n=\lambda f_n$ on $U$ if and only if
 \begin{equation} \label{eq:3.2}
 4 \left(1-t\right)^2 \left[ t y_n''(t)+(|n|+1) y_n'(t) \right]=\lambda y_n(t)
 \end{equation}
 for all $t\in \D_{r^2}$. A power series ansatz shows that there is at most one solution $y_n$ of (\refeq{eq:3.2}) which is holomorphic in a neighborhood of $t=0$ and normalized such that $y_n(0)=1$. In order to find this solution, we convert
\eqref{eq:3.2} into a hypergeometric differential equation as follows. We choose $\mu \in \C$ such that $\lambda=4 \mu(\mu-1)$ and let $\hat{y}_n(t):=(1-t)^{-\mu} y_n(t)$. Then (\refeq{eq:3.2}) is equivalent to
\begin{equation} \label{eq:3.2New}
t (1-t) \hat{y}_n''+\big[c-(a+b+1) t\big] \hat{y}_n'-a b  \hat{y}_n=0
\end{equation}
with
$$ a=\mu, \, b=\mu+|n|, \, c=|n|+1 \, .$$
It is well--known (see \cite[\S 15.10]{dlmf}) that the only solution $\hat{y}_n$ of (\refeq{eq:3.2New}) which is holomorphic at $t=0$ and normalized at $t=0$ by  $\hat{y}_n(0)=1$ is the  hypergeometric series
\begin{equation*}\label{eq:3.3}
		\Fz{a,b}{c}{t}:=\sum_{k=0}^\infty\frac{(a)_k(b)_k}{(c)_k}\frac{t^k}{k!}\,,\quad |t|<1,
	\end{equation*}
	where
	\begin{equation}\label{eq:Pochhammer}
		(\alpha)_k:=\prod\limits_{j=0}^{k-1}(\alpha+j), \, \qquad \alpha\in\C, \, k\in\N_0
	\end{equation}
	denotes the \emph{(rising) Pochhammer symbol}. This procedure leads to the
	conclusion that \begin{equation}\label{eq:eigenvaluehypfunction}
		y_{n}(t)=(1-t)^\mu\Fz{\mu,\mu+|n|}{|n|+1}{t}
	\end{equation}
   is the unique solution of (\refeq{eq:3.2}) which
	is holomorphic in $t=0$ and  normalized by $y_n(0)=1$. Note that
 there are in fact two complex numbers $\mu \in \C$ such that $\lambda=4 \mu(\mu-1)$. However, as we have seen, both necessarily lead to the same holomorphic solution $y_n$ of (\refeq{eq:3.2}) with $y_n(0)=1$.

   \begin{remark}\label{rem:Poissonplusminus}
       Note that if $\mu$ is one solution to $\lambda=4\mu(\mu-1)$, then $1-\mu$ is the other one. As we have seen, both numbers induce the same function \eqref{eq:eigenvaluehypfunction}. This also follows from the transformation formula (see \cite[Eq. 15.3.3]{abramowitz1984})
       \begin{align*}
			\Fz{a,b}{c}{t}
			&=
			(1-t)^{c-a-b}
			\Fz{c-a,c-b}{c}{t} \, .
		\end{align*}
    In the following it will often be convenient to choose $\mu$ such that $\Re\mu\geq1/2$.
   \end{remark}

  By a standard fact about hypergeometric functions (\cite[15.3.1]{abramowitz1984}) the function $y_n$ has a  holomorphic extension at least to the slit plane
  $$ \C \setminus [1,\infty)\, . $$ Returning to \eqref{eq:3.1New} with this choice of $y_n$, we therefore see that
  \begin{equation} \label{eq:F_Def}   F^{\mu}_n(z,w):=  (1-zw)^\mu \Fz{\mu,\mu+|n|}{|n|+1}{zw} \cdot \begin{cases} z^n & \text{ if } n \ge 0 \, , \\ w^{|n|} & \text{ if } n \le 0 \, ,\end{cases}
  \end{equation}
    is holomorphic at least on the domain
    $$ \Omega_{*}:=\left\{(z,w) \in \C^2 \, : \, zw \in \C \setminus [1,\infty)\right\} $$
    and provides the unique $n$--homogeneous solution of $\Delta_{zw} f=\lambda f$ on $\Omega_*$ up to a multiplicative constant.

    \medskip

    Next, we relate the $n$--homogeneous eigenfunction $F^{\mu}_n$ of $\Delta_{zw}$ to a complexified version of the Poisson kernel of the unit disk. This slight change of perspective will turn out to be important in the sequel. In fact, a well--known integral representation formula for the hypergeometric function $\Fz{\mu,\mu+|n|}{|n|+1}{zw}$, see \cite[Section 2.5.1, Formula (10), p.~81]{Erdelyi1953}, shows that for every $n \in \Z$,
    \begin{equation} \label{eq:FNvsPM}
    \binom{\mu+|n|-1}{|n|} F^\mu_n(z,w)=\frac{(1-z w)^{\mu}}{2\pi} \int \limits_{0}^{2\pi} \frac{e^{in t}}{\left( 1-z e^{-it} \right)^{\mu} \left(1-w e^{it} \right)^{\mu}} \, dt \, ,
    \end{equation}
    with the \emph{generalized binomial coefficients}
       $ \binom{\mu}{n}
        \coloneqq
        \frac{(\mu - n +1)_n}{n!}
        $ for $\mu \in \C$, $ n \in \N_0$.
    On the right--hand side of this identity one can recognize the $(-n)$--th Fourier coefficient of (a suitably defined power of order $\mu$ of) the \emph{generalized Poisson kernel}
	\begin{equation}
		\label{eq:generalizedPoisson}
		P:\D^2 \times \partial \D \rightarrow\C\,,\quad P(z,w;\xi)
		:=
		\frac{1-zw}
		{(1-z/\xi)(1-w\xi)}\,.
	\end{equation}

 Note that if $w=\cc{z}\in \D$, then $P(z,\overline{z};\xi)$ is the standard Poisson kernel of the unit disk. These considerations motivate the following definition.
	\begin{definition}[Poisson Fourier mode, PFM]
		Let $\mu\in\C$ and $n\in\Z$. Then the function
		\begin{equation}
			\label{eq:PoissonFourier}
			P_n^{\mu}: \D^2
   \rightarrow\C\,,\quad P_n^\mu(z,w)
			:=
			\frac{\left(1-zw\right)^{\mu}}{2\pi }\int\limits_{0}^{2\pi} \frac{e^{-int}}{\left(1-z e^{-it}\right)^\mu \left(1-w e^{it} \right)^\mu} \, dt\, .
		\end{equation}
 is called the \emph{$n$--th Poisson Fourier mode (PFM) of order $\mu$}.\\ Here, $a^\mu$ is defined for $a \in \C \setminus (-\infty,0]$ as $\exp\left( \mu \log a\right)$, where $\log$ denotes the principal branch of the logarithm.
 \end{definition}
\begin{remark}\label{rem:whenPFMindeedisPoissonkernelpowerFouriermode}
    If $w=\overline{z} \in \D$, then $P^\mu_n(z,\overline{z})$ is the $n$--th Fourier coefficient of the $\mu$--power of the (real--valued and, in fact, non--negative) Poisson kernel of the unit disk $\D$. Further, if $\mu=m\in\Z$, then $P^m_n(z,w)$ also is the $n$--th Fourier coefficient of the $m$--power of the generalized Poisson kernel from \eqref{eq:generalizedPoisson}.
\end{remark}

We can now reformulate (\ref{eq:FNvsPM}) in terms of PFMs as follows:
     \begin{equation} \label{eq:PMvsF}
                          P^\mu_n=(-1)^n \binom{-\mu}{|n|}  F^\mu_{-n} =\binom{\mu+|n|-1}{|n|} F^\mu_{-n}\, .
                        \end{equation}
    In particular, $P^\mu_{n}$ has a holomorphic extension from $\D^2$ to $\Omega_*$, which we continue to denote by~$P^{\mu}_n$.

    In analogy with \eqref{eq:3.1} we call a subdomain $D \subseteq \Omega$ \emph{rotationally invariant} if $(\eta z, w/\eta) \in D$ for all $(z,w) \in D$ and all $\eta \in \del \D$.
    Summarizing our considerations leads to the following complete description of the $n$--homogeneous eigenfunctions of $\Delta_{zw}$:
    \begin{theorem} \label{thm:helgason}
    Let $D$ be a rotationally invariant subdomain of $\Omega$ containing the origin and  $n\in\Z$. Suppose that  $f \in \mathcal{H}(D)$ is an $n$--homogeneous solution to $\Delta_{zw}f=\lambda f$ for some $\lambda \in \C$ of the form $\lambda=4\mu(\mu-1)$ with $\Re \mu \ge 1/2$.  Then there is a constant $c \in \C$ such that
				\begin{equation*}\label{eq:Helgasonlemma1}
					f(z,w)
					=
					c \, P_{-n}^{\mu}(z,w),
					\qquad (z,w) \in D \cap \Omega_*\, .
				\end{equation*}
				In particular, $f$ has a holomorphic extension to $D \cup \Omega_*$.
    Moreover, the following dichotomy holds:
                                				\begin{itemize}
					\item[(NE)] (Non--exceptional case)\\
     If $\lambda\not=4m (m+1)$ for all $m \in \N_0$, then $\Omega_*$ is the maximal domain of existence of $f$.
					\item[(E)] (Exceptional case)\\ If $\lambda=4m(m+1)$ for some $m \in \N_0$, then
					the maximal domain of existence of $f$ is
$$\begin{array}{lcl}
         \Omega_+ & &    n < -m\,;   \\[2mm]
                \Omega_-  & \text{ if }\quad &  n >m \,;  \\[2mm]
         \Omega & &   \! |n|\le m\,.   \end{array}$$

				\end{itemize}
                          \end{theorem}

  \begin{proof} It remains to prove that for any nonconstant $n$--homogeneous eigenfunction $F^{\mu}_n$ of $\Delta_{zw}$ in $\mathcal{H}(\Omega_*)$ the dichotomy ``(NE) vs.~(E)'' holds.

  \medskip

  (i)
      Let $\mu \not \in \N$ and assume that $\Omega_*$ is not the maximal domain of existence of $F^\mu_n$ which is contained in $\Omega$, see Definition \ref{def:MaximalDomain}. Then $F^\mu_n$ has a holomorphic extension to some point $(z_0,w_0) \in \Omega$ such that $z_0w_0 \in \R \cup \{\infty\}$ and $z_0w_0 > 1$. We first consider the case $n \ge 0$. By definition of $F^\mu_n$, see  (\ref{eq:F_Def}), and in view of \cite[15.3.4]{abramowitz1984}, we have
\begin{equation} \label{eq:2F1MT}
  F^{\mu}_n(z,w)=(1-zw)^{\mu}\Fz{\mu,\mu+n}{|n|+1}{zw}z^n =\Fz{\mu,1-\mu}{|n|+1}{\frac{zw}{zw-1}} z^n\, ,
  \end{equation}
  from which we infer that
  $$t \mapsto \Fz{\mu,1-\mu}{|n|+1}{t}$$
  is holomorphic in a neighborhood of the point $x_0:=z_0w_0/(z_0w_0-1) \ge 1 $. However, see \cite[15.2.3]{dlmf},
  \begin{equation} \label{eq:2F1jump} \begin{split} \lim \limits_{y \to 0^+} & \big[ \Fz{\mu,1-\mu}{|n|+1}{x+i y} -\Fz{\mu,1-\mu}{|n|+1}{x-i y}\big]\\ & \quad =\frac{2\pi i}{\Gamma(\mu) \Gamma(1-\mu)} \left(x-1 \right)^{|n|} \Fz{|n|+1-\mu,|n|+\mu}{|n|+1}{1-x} \end{split}
  \end{equation}
  for all $x >1$. By our assumption, the left--hand side of (\ref{eq:2F1jump}) has to vanish for all $x \in \R$ in some open interval $(x_0,x_0+\eps)$ with $\eps>0$ and hence the same is true for the right--hand side. Since the right--hand side is a holomorphic function of $x$ on the domain $\C \setminus (-\infty,0]$ this is clearly only possible if $\mu \in \N$. The remaining case $n \le 0$ follows from the case $n \ge 0$ and $F^\mu_{-n}(z,w)=F^{\mu}_n(w,z)$.

  \medskip

  (ii)  Let $\mu \in \N$ and $|n| \le \mu-1$.  If $n \ge 0$, then
  $$ F^\mu_n(z,w)=\frac{G^\mu_n(zw)}{(1-zw)^{\mu-1}} z^{n} \, ,$$
  where
  \begin{equation} \label{eq:2F1symm}
  \begin{aligned}
   G^{\mu}_n(zw)&:= (1-zw)^{2 \mu-1}   \Fz{\mu,\mu+|n|}{|n|+1}{zw}\\ &\hspace{0.11cm}=\Fz{1-\mu,|n|+1-\mu}{|n|+1}{zw}
   \end{aligned}
\end{equation}
is a  polynomial in $zw$
of degree $\mu-n-1 \ge 0$, see \cite[15.3.3 and 15.1.1]{abramowitz1984}. Hence $F^\mu_n$ is the product of $z^n$ and  a rational function in $zw$ of numerator degree $\mu-n-1$ and of denominator degree $\mu-1$ with pole only at  the point $1$. Therefore, $F^\mu_n \in \mathcal{H}(\Omega)$. Since  $F^{\mu}_n(z,w)=F^{\mu}_{-n}(w,z)$, this implies $F^{\mu}_n \in \mathcal{H}(\Omega)$ also for $1-\mu \le n \le 0$.

\medskip

(iii) Let $\mu \in \N$ and $|n| \ge \mu$. Since  $F^{\mu}_n(z,w)=F^{\mu}_{-n}(w,z)$ and $\mathcal{H}(\Omega_+) \cap \mathcal{H}(\Omega_-)=\mathcal{H}(\Omega)$ we may assume $n \ge \mu$. In this case, the function $G^\mu_n$ in (\ref{eq:2F1symm}) is a polynomial in $zw$ of degree $\mu-1$, and thus  $F^\mu_n$  is the product of $z^{n}$ times a rational function in $zw$ of numerator degree $ \mu-1$ and of denominator degree $\mu-1$, and is therefore holomorphic on $\Omega_-$. Moreover, $F^\mu_n$ has no holomorphic extension to a point $(\infty,w_0)$ with $w_0 \in \widehat{\C}\setminus\{0\}$ since in view of (\ref{eq:2F1symm}) and (\ref{eq:2F1MT})
\begin{align*} \lim \limits_{z \to \infty} \frac{G^\mu_n(zw_0)}{(1-zw_0)^{\mu-1}} & =\lim \limits_{z \to \infty} \Fz{\mu,1-\mu}{|n|+1}{\frac{zw_0}{z w_0-1}}=\Fz{\mu,1-\mu}{|n|+1}{1}\\
  &= \frac{\Gamma(n+1) \Gamma(n)}{\Gamma(n+1-\mu) \Gamma(n+\mu)} \in (0,\infty)\, ,
\end{align*}
by \cite[15.1.20]{abramowitz1984},
so $|F^\mu_n(z,w_0)| \to \infty$ as $z \to \infty$. This implies that $\Omega_-$ is the maximal domain of existence of $F^\mu_n$.
\end{proof}

 \begin{corollary}\label{cor:F}
     For every $\lambda\in \C$ and each $n \in \Z$ there is an $n$--homogeneous holomorphic solution of $\Delta_{zw}f =\lambda f$ on the domain $\Omega_*$. More precisely, if $\lambda=4\mu(\mu-1) \in \C$ with $\Re \mu \ge 1/2$, then every such solution has the form $c P^{\mu}_{-n}$ for some $c \in \C$.
 \end{corollary}

We close this section by collecting some elementary properties of the PFM $P^\mu_n$ which will be needed in the sequel.

\begin{remark}[Elementary properties of Poisson Fourier modes]\label{rem:PFMproperties}
Let $\mu \in \C$, $n \in \N_0$ and $z,w \in \D$.
\begin{enumerate}[(a)]

\item The Poisson Fourier modes are related to the hypergeometric function $\F$ via
			\begin{align}
   \begin{split}\label{eq:PoissonFourierHypergeometric1}
				P_{n}^{\mu}(z,w)
				&=
				(-1)^n\binom{-\mu}{n}
				(1-zw)^\mu w^n
				\Fz{\mu,\mu+n}{n+1}{zw}\\
				&=
				(-1)^n\binom{-\mu}{n}
				 w^n
				\Fz{\mu,1-\mu}{n+1}{\frac{zw}{zw-1}}\,.
                \end{split}
			\end{align}
   This is (\ref{eq:F_Def}) together with (\ref{eq:PMvsF}) resp.~(\ref{eq:2F1MT}).

\item Using the series representation of $\F$ functions we see that
\begin{equation}\label{eq:generalPoissonfouriermode}
				P_{n}^{\mu}(z,w)
				=
				(-1)^n(1 - zw)^\mu
				\sum_{k=0}^{\infty}
				\binom{-\mu}{k + n}
				\binom{-\mu}{k}
				z^{k} w^{k+n}
    \,.
			\end{equation}
If $-\mu=m \in \N_0$, then $P_{\pm n}^{-m} = 0$, whenever $n > m$. Otherwise,
			\begin{equation}
				\label{eq:PoissonFourierExplicit}
				P_{n}^{-m}(z,w)
				=
				(-1)^n
				\sum_{k=0}^{m-n}
				\binom{m}{k + n}
				\binom{m}{k}
				\frac{z^{k} w^{k+n}}{(1 - z w)^m}
    \,.
			\end{equation}

   \item The PFM are symmetric in the sense that $P_{n}^{\mu}(z,w)=P_{-n}^{\mu}(w,z)$. This allows us to simplify our proofs in the following: we will often prove identities for $P^{\mu}_{n}$ only which then implies the corresponding result for $P^{\mu}_{-n}$.

   \item\label{item:Poissonplusminus} Remark \ref{rem:Poissonplusminus} implies
   \begin{equation}\label{eq:Poissonplusminus}
       \binom{\mu-1}{n}P_{\pm n}^{\mu}=\binom{-\mu}{n}P_{\pm n}^{1-\mu}\,.
   \end{equation}
		\end{enumerate}

\end{remark}
\begin{remark}[Invariant representative functions and the finite dimensional invariant eigenspaces]
		We consider $\D$ as a symmetric space $\mathcal{M}(\D) / \partial \D$ over its automorphism group $\mathcal{M}(\D)$ with $\partial \D \cong \operatorname{U}(1)$ acting by rotations. This yields an alternative way of deriving the restrictions to $\D$ of those Poisson Fourier modes from \eqref{eq:PoissonFourierExplicit}, which are defined on all of $\Omega$, using finite dimensional representation theory. Such considerations are very much in the spirit of \cite{Helgason70}. The group $\mathcal{M}(\D)$ is isomorphic to the projective split unitary group $$\operatorname{PSU}(1,1) = \Big\{
            \mbox{$\scriptsize
            \begin{pmatrix}
				a & b \\
				\cc{b} & \cc{a}
            \end{pmatrix}$}
            \colon
            \abs{a}^2-\abs{b}^2=1
            \Big\}
            \Big/
            \Big\{
            \pm
            \mbox{$\scriptsize
            \begin{pmatrix}
				1 & 0 \\
				0 & 1
			\end{pmatrix}$}
            \Big\}$$ as a real Lie group.
            By Schur's Lemma every irreducible representation $\pi \colon \mathcal{M}(\D) \longrightarrow \operatorname{GL}_n(\C)$ as invertible $(n\times n)$--matrices induces eigenfunctions of the Laplacian on $\operatorname{PSU}(1,1)$ via
			\begin{equation}
				\label{eq:RepFunction}
				\pi_{k, \ell}
				\colon
				\operatorname{PSU}(1,1) \longrightarrow \C,
				\quad
				\pi_{k, \ell}(g)
				\coloneqq
				\pi(g)_{k, \ell}
				\, ,
				\qquad 1 \le k, \ell \le n,
			\end{equation}
			which are known as \emph{representative functions} or \emph{matrix elements}. For a systematic discussion we refer to the textbook \cite[Chapter~III]{BroeckerDieck2003}. The corresponding eigenvalues may be computed from representation theoretic data, see e.g. \cite[Prop.~10.6]{Hall2010}. Note that in the literature, many results are formulated for general invariant differential operators or the corresponding Casimir elements and joint eigenfunctions thereof instead of the Laplacian, which is always invariant and thus constitutes a special case. However, the disk $\D$ is a two--point homogeneous space, so all $\mathcal{M}(\D)$--invariant differential operators are polynomials in the Laplacian, see \cite[Theorem~11]{Helgason1959}. Adapting \cite[Example~4.10]{Hall2010}, one may parametrize the irreducible representations of $\operatorname{PSU}(1,1)$ as follows: let $m \in \N_0$ and
			\begin{equation*}
				V_m
				\coloneqq
				\operatorname{span}
				\big\{
				z^k w^{m-k}
				\in
				\C[z,w]
				\;\big|\;
				0 \le k \le m
				\big\}
			\end{equation*}
			the vector space of polynomials of total degree $m$ with $\dim(V_m) = m+1$. Then
			\begin{equation*}
				\label{eq:SUIrreps}
				\pi_m
				\colon
				\operatorname{SU}(1,1)
				\longrightarrow
				\operatorname{GL}(V_m), \quad
				\pi_m
                \mbox{$\scriptsize
				\begin{pmatrix}
				a & b \\
				\cc{b} & \cc{a}
                \end{pmatrix}$}
				p
				(z,w)
				\coloneqq
				p
				\big(
                    \cc{a}z - bw,
                    -\cc{b}z + aw
                \big)
			\end{equation*}
			for $m \in \N$ constitutes a complete list of the irreducible finite dimensional representations of $\operatorname{SU}(1,1)$, see also \cite[Prop.~4.11 \& Sec.~4.6]{Hall2010}. Moreover, $\pi_m$ descends to the quotient $\operatorname{PSU}(1,1)$ if and only if $m$ is even, providing a description of all irreducible finite dimensional representations of $\operatorname{PSU}(1,1)$. A computation then yields explicit formulas for \eqref{eq:RepFunction}, which completes the eigenvalue theory of the Laplacian on $\operatorname{PSU}(1,1)$. Finally, parametrizing a copy of the rotation group $\operatorname{U}(1) \subseteq \operatorname{PSU}(1,1)$ by $\big(\begin{smallmatrix}
				i \eta & 0 \\
				0 & - i \cc{\eta}
			\end{smallmatrix}\big)$ with $\eta \in \partial \D$ yields that exactly the representative functions $(\pi_{2m})_{m,d}$ with $1 \le d \le 2m+1$ are invariant under the action of $\operatorname{U}(1)$ and thus pass to the quotient $\D \cong \operatorname{PSU}(1,1) / \operatorname{U}(1)$. By a computation, $(\pi_{2m})_{m,d} = P_{m-d}^{-m}$ with the exceptional Poisson Fourier modes from \eqref{eq:PoissonFourierExplicit}. Note that this only recovers the finite dimensional invariant eigenspace, i.e. the case $(E_0)$ in Theorem~\ref{thm:RudinIntro}. It would be interesting to study whether this approach generalizes to the other invariant subspaces by incorporating representations on infinite dimensional spaces.
		\end{remark}

		\section{Spectral decomposition of eigenspaces}
        \label{sec:Helgason}

			In this section we show that for every rotationally invariant domain $D\subseteq \Omega$ containing the origin each holomorphic eigenfunction of the invariant Laplacian $\Delta_{zw}$ on $D$ has a unique representation as a \textit{Poisson Fourier series}, a  doubly infinite series with Poisson Fourier modes as building blocks. We shall also see that if $D=\Omega$ this series corresponds to a finite sum, and when~$D$ is one of the distinguished domains $\Omega_+$ or $\Omega_-$, then the series is one--sided infinite.

			\begin{theorem}\label{thm:Helgason}
				Let $D$ be a rotationally invariant subdomain of $\Omega$ containing the origin, and let $f \in \mathcal{H}(D)$ be such that $\Delta_{zw} f=\lambda f$ for some $\lambda \in \C$ of the form $\lambda=4\mu(\mu-1)$ with $\Re \mu \ge 1/2$. Then there are uniquely determined coefficients $c_n \in \C$ such that
				\begin{equation} \label{eq:PMSeries}
					f
					=\sum \limits_{n=-\infty}^{\infty} c_n P_n^{\mu}:=
					\sum_{n = 0}^\infty
					c_n
					P_n^{\mu}
					+
					\sum_{n = 1}^\infty
					c_{-n}
					P_{-n}^{\mu}\,.
				\end{equation}
				Here, both series converge absolutely and locally uniformly in $D$.

			\end{theorem}

			\begin{proof}
				(i) For each $n \in \Z$ and all $(z,w) \in D$ we consider
				$$ f_n(z,w):=\frac{1}{2\pi i} \int \limits_{\partial \D} f \left( \eta z,\frac{w}{\eta} \right) \eta^{-(n+1)} \, d\eta \, .$$
				Since $D$ is rotationally invariant, $f_n$ is well--defined. Clearly, $f_n$ is holomorphic on $D$  and $n$--homogeneous. By Theorem \ref{thm:helgason}, there are complex numbers $c_n \in \C$ such that
				$$ f_n(z,w)=c_n P^{\mu}_{-n}(z,w) \, . $$

				(ii) We fix $(z,w) \in D$ and  choose positive constants $r<1<R$ such that $(\eta z,w/\eta) \in D$ for all $r<|\eta|<R$. This is possible as $D$ is a rotationally invariant domain. Then
				$\eta \mapsto f(\eta z,w/\eta)$ is holomorphic in the annulus $r<|\eta|<R$ and therefore has a representation as the Laurent series
				$$ f\left(\eta z,\frac{w}{\eta} \right)=\sum \limits_{n=-\infty}^{\infty} f_n(z,w) \eta^n \, ,$$
                which converges locally uniformly in $r<|\eta|<R$.			In particular,
				$$ f\left(z,w \right)=\sum \limits_{n=-\infty}^{\infty} f_n(z,w)  \, , \qquad (z,w) \in D \, .$$
				This  series converges in fact uniformly on each compact set $K\subseteq D$. In order to see this, let $K$ be such a compact subset of $D$. We can then choose  positive constants $r_1<1<R_1$ such that ${K_1:=\{(\eta z,w/\eta) \, : \, (z,w) \in K, \, r_1 \le |\eta| \le R_1\}}$ is a compact subset of $D$, and we let
				${M_1:=\max\{ |f(z,w)| \, : (z,w) \in K_1\}}$. Note that
				$$f_n(z,w)=\frac{1}{2\pi i} \int \limits_{\partial\D_{\rho}} f\left(\eta z,\frac{w}{\eta} \right) \, \eta^{-(n+1)} \, d \eta \, $$
				for all $n \in \Z$, all $(z,w) \in K$ and every $\rho \in [r,R]$. We therefore get
				$$ |f_n(z,w)| \le M_1 \cdot \rho^{-n} \, \quad \text{ for all } (z,w) \in K \text{ and all } \rho  \in [r_1,R_1] \, .$$
				In particular, $|f_n(z,w)| \le M_1 r_1^{-n}$ for all $n < 0$ and $|f_n(z,w)| \le M_1 R_1^{-n}$ for all $n \ge 0$, and this ensures the absolute and uniform convergence of the two series
				$$ \sum \limits_{n=-\infty}^{-1} f_n(z,w) \quad \text{ and } \quad \sum \limits_{n=0}^{\infty} f_n(z,w) $$
				on the compact set $K$.

				(iii) In view of Step (ii) the coefficients $f_n(z,w)$ are exactly the Laurent coefficients of ${\eta \mapsto f(\eta z,w/\eta)}$
				in an annulus containing the unit circle and are thus uniquely determined by $f$. Hence, $f_n(z,w)=c_n P_{-n}^{\mu}(z,w)$ shows that the coefficients $c_n$ are uniquely determined by~$f$.
			\end{proof}

			In fact, the previous proof provides the following more precise information.

			\begin{corollary} \label{cor:EW}
				Let $D$ be a subdomain of $\Omega$ containing the origin, and let $f \in \mathcal{H}(D)$ be such that $\Delta_{zw} f=\lambda f$ for some $\lambda\in\C$.
		\begin{itemize}
		\item[(i)] If $D=\Omega$, then  $\lambda=4m (m+1)$ for some $m\in\N_0$ and there are uniquely determined coefficients $c_n\in \C$ such that
					$$ f=\sum \limits_{n=-m}^m c_n P^{m+1}_n \, .$$
					\item[(ii)] If $D=\Omega_+$, then  $\lambda=4m (m+1)$ for some $m\in\N_0$ and there are uniquely determined coefficients $c_n \in \C$ such that
					$$ f=\sum \limits_{n=-\infty}^m c_n P^{m+1}_n \, .$$
					\item[(iii)] If $D=\Omega_-$, then  $\lambda=4m (m+1)$ for some $m\in\N_0$ and there are uniquely determined coefficients $c_n \in \C$ such that
					$$ f=\sum \limits_{n=-m}^{\infty} c_n P^{m+1}_n \, .$$
				\end{itemize}
			\end{corollary}

			\begin{proof}
				(i) If $f\in\mathcal{H}(\Omega)$, then the functions
				$$ f_n(z,w):=\frac{1}{2\pi i} \int \limits_{\partial \D} f \left( \eta z,\frac{w}{\eta} \right) \eta^{-(n+1)} \, d\eta \, $$
				are holomorphic  and $n$--homogeneous on $\Omega$ and
				$$ f_n(z,w)=c_n P^{\mu}_{-n}(z,w) \, $$
                                with $\mu \in \C$ such that $\lambda=4\mu(\mu-1)$ and $\Re\mu \ge 1/2$.
				Hence, $\lambda= 4 m (m+1)$ for some $m \in \N_0$ since otherwise $P^{\mu}_{-n}$ is not holomorphic on $\Omega$ by Theorem \ref{thm:helgason}. If $\lambda=4m(m+1)$, then for each $|n|> m$ the function $P^{m+1}_{-n}$ is not holomorphic on $\Omega$ again by Theorem \ref{thm:helgason} which forces $c_n=0$ for those $n$. Parts (ii) and (iii) follow in the same way.
			\end{proof}
Since all PFM $P^{\mu}_n$ are holomorphic on $\Omega_*$, it is  natural to inquire whether the series (\ref{eq:PMSeries}) in Theorem \ref{thm:Helgason} converges on some bigger domain than $D$. In Section \ref{sec:CompareRudin Helgson} we shall see that in general this is not the case.

  \section{Comparison with the eigenvalue equation of the Laplacian on the unit disk and the Riemann sphere}\label{sec:CompareRudin Helgson}\

In this section we relate the spectral theory of the Laplacian on the unit disk $\D$ resp.~the Riemann sphere $\widehat{\C}$ with the spectral theory of the invariant Laplacian on $\Omega$ which we have developed so far. In particular, we show
that all eigenfunctions of $\Delta_{\D}$ on $\D$ resp.~$\Delta_{\widehat{\C}}$ on~$\widehat{\C}$
do have holomorphic extensions to eigenfunctions of $\Delta_{zw}$ on the bidisk $\D^2$ resp.~$\Omega$. Our approach is similar to the one employed in \cite[Section 1.6]{BerensteinGay1995} which deals exclusively with the hyperbolic Laplacian $\Delta_{\D}$ on the unit disk $\D$. However, we need some fine properties of hypergeometric functions, in addition to those which have been employed in \cite{BerensteinGay1995}. We begin with the following lemma which is crucial for our approach.

\begin{lemma} \label{lem:2F1Asymptotic}
  Let $\mu \in \C$ with $\Re \mu>0$. Then
  \begin{equation} \label{eq:2F1Asymptotics2}
    \lim \limits_{n \to \infty}
    \Fz{\mu,1-\mu}{n+1}{\omega}
    =
    1
    \end{equation}
  locally uniformly for $\omega \in \C \setminus [1,\infty)$.
\end{lemma}
  \begin{proof} Fix $m \in \N_0$ such that $m> \Re \mu-2$ and let $n \in \N_0$.
    Then by \cite[Formula (11), p.~76]{Erdelyi1953} or \cite[p.~84]{MacRobert} there are complex numbers $A_k$, $k=1, \ldots, m$, such that
    \begin{equation*} \label{eq:AsympExpansion}
     \Fz{\mu,1-\mu}{n+1}{\omega}
     =
     1+\sum \limits_{k=1}^m \frac{A_k}{\left( n+1 \right) \ldots \left(n+k \right)} \frac{\omega^k}{k!}+\rho_{m+1}(n,\omega) \, ,
    \end{equation*}
    where
    \begin{equation} \label{eq:err1}
     \rho_{m+1}(n,\omega):=\tilde{\gamma}_n
     \int \limits_{0}^1 \int \limits_0^1 t^{1-\mu+m} (1-t)^{n+\mu-1} \left( 1- s t \omega\right)^{-\mu-1-m} (1-s)^m \, ds \, dt \, \cdot \, \omega^{m+1}
     \end{equation}
     and
     \begin{equation*} \label{eq:err2}
       \tilde{\gamma}_n:=\frac{\Gamma(n+1) \Gamma(\mu+m)}{\Gamma(1-\mu) \Gamma(\mu+n) \Gamma(\mu) \, m!} \, .
       \end{equation*}
Note that our choice of the non--negative integer $m$ guarantees that the integral in (\ref{eq:err1}) converges.
To prove (\ref{eq:2F1Asymptotics2})  it therefore suffices to show that $\rho_{m+1}(n,\omega) \to 0$  uniformly on every compact subset $K$ of $\C \setminus [1,\infty)$ as $n \to \infty$. Fix such a compact set $K$. Obviously,
$$ M_K:=\min \limits_{\omega \in K,0 \le s,t \le 1} |1-st \omega|>0 \, ,$$
and since $-\Re\mu-1-m<0$, we have
\begin{align} \nonumber
  |\rho_{m+1}(n,\omega)| & \le |\tilde{\gamma}_{n}| \int \limits_{0}^1 t^{1-\Re\mu+m} (1-t)^{n+\Re\mu-1} \, dt \frac{|\omega|^{m+1}}{M_K^{m+1+\Re\mu}} \\ &
    \nonumber                           =
    |\tilde{\gamma}_n| \frac{\Gamma(2+m-\Re\mu) \, \Gamma(n+\Re\mu)}{\Gamma(2+m+n)} \frac{|\omega|^{m+1}}{M_K^{m+1+\Re\mu}} \\ &=
    \label{eq:st1}                       \gamma^*                            \frac{\Gamma(n+\Re\mu)}{\Gamma(n+1)} \frac{\Gamma(n+1)}{|\Gamma(n+\mu)|} \frac{\Gamma(n+1)}{\Gamma(2+m+n)}  \frac{|\omega|^{m+1}}{M_K^{m+1+\Re\mu}}
\end{align}
 where
$$\gamma^*:=                         \frac{|\Gamma(\mu+m)| \, \Gamma(2+m-\Re\mu)}{|\Gamma(1-\mu)| \, |\Gamma(\mu)| \, m!} \, .$$
An application of Stirling's formula
\begin{equation} \label{eq:Stirling}
  \lim \limits_{n \to \infty} \frac{\Gamma(\alpha+n)}{\Gamma(n+1)} e^{-(\alpha-1) \log(n+1)} \to 1 \, , \quad \alpha \in \C \, ,
  \end{equation}
to each of the first three quotients in (\ref{eq:st1}) shows that there is a constant
$\gamma>0$ depending only on $\mu$, $m$ and $K$  such that
$$  |\rho_{m+1}(n,\omega)| \le \gamma \cdot \left(n+1 \right)^{\Re\mu-1}  \left(n+1 \right)^{1-\Re\mu}   \left(n+1 \right)^{-m-1} |\omega|^{m+1}  =\gamma  \cdot \left( \frac{|\omega|}{n+1} \right)^{m+1} \, .$$
In particular, $\rho_{m+1}(n,\omega) \to 0$ as $n \to \infty$ uniformly for $\omega \in K$. This completes the proof of (\ref{eq:2F1Asymptotics2}).
        \end{proof}

We are now in a position to prove Theorem \ref{thm:DiskvsBidiskIntro}.

\begin{proof}[Proof Theorem \ref{thm:DiskvsBidiskIntro}]
Let $f \in X_\lambda(\D)$, and write $\lambda=4 \mu(\mu-1)$ for some complex number $\mu \in \C$ with $\Re\mu \ge 1/2$.	By \cite[Theorem 16.18]{BerensteinGay1995} there are uniquely determined coefficients $c_n \in \C$ such that
			$$f(z)=\sum \limits_{n=-\infty}^{\infty} c_n P^{\mu}_n(z,\overline{z}) \, ;$$
the series converges absolutely and pointwise for each $z \in \D$. As it is shown in the proof  of \cite[Theorem 1.6.18]{BerensteinGay1995}, the coefficients $c_n$ do have the additional property that
			\begin{equation} \label{eq:BerensteinGay1}
				\sum \limits_{n=-\infty}^{\infty} |c_n| r^{|n|}<\infty \quad \text{ for all } 0<r<1 \, .
			\end{equation}
			We proceed to show that (\ref{eq:BerensteinGay1}) and Lemma \ref{lem:2F1Asymptotic} together guarantee   that the series
			\begin{equation} \label{eq:BerensteinGay2}
				F(z,w):=\sum \limits_{n=-\infty}^{\infty} c_n P^{\mu}_n(z,w)
			\end{equation}
			converges locally uniformly for $(z,w) \in \D^2$, and hence defines a function ${F \in \mathcal{H}(\D^2)}$ with the required properties. The identity principle shows further that $F$ is then uniquely determined.

\medskip
        It remains to prove the local uniform convergence of the series (\ref{eq:BerensteinGay2}) in $\D^2$. Fix $r \in (0,1)$.  We begin by noting that (\ref{eq:PMvsF}) and (\ref{eq:F_Def}) lead to
        \begin{equation} \label{eq:PMestimate1} \begin{split}
            \left| P^{\mu}_n(z,w) \right| & = \left| \binom{\mu}{|n|} \right| \, \left|F^{\mu}_{-n}(z,w) \right| \le
            \frac{|\Gamma(\mu +|n|)|}{|\Gamma(\mu)| \, \Gamma(|n|+1)} \left| \Fz{\mu,1-\mu}{|n|+1}{\frac{zw}{zw-1}} \right|\, r^{|n|}
        \end{split}
        \end{equation}
                            for all $ |z|,|w| \le r$.
                        Since the M{\"o}bius transformation
                        $$ \xi \mapsto \frac{\xi}{\xi-1}$$ maps the unit disk $\D$ conformally onto the half--plane $\{\zeta \in \C \, : \, \Re \zeta<1/2\}$, the set
                        $$ K:=\left\{ \frac{zw}{zw-1} \, : \, |z| \le r, \, |w| \le r \right\}$$
                        is a compact subset of $\C \setminus [1,\infty)$. Thus, Lemma \ref{lem:2F1Asymptotic} implies
    $$ \Fz{\mu,1-\mu}{|n|+1}{\frac{zw}{zw-1}} \to 1 \quad \text{ uniformly for } |z|\le r, \, |w| \le r\, $$
    as $|n| \to \infty$.
Combining this with  Stirling's formula (\ref{eq:Stirling}) we see from inequality (\ref{eq:PMestimate1})   that there is a constant $\gamma>0$ such that
\begin{equation} \label{eq:PMEstimate}
  \left| P^{\mu}_n(z,w)\right|  \le \gamma \cdot \left(|n|+1 \right)^{\Re\mu-1} \, r^{|n|}
  \end{equation}
  for all $|z| \le r$, $|w| \le r$ and every $n \in \Z$. 	This estimate together with (\ref{eq:BerensteinGay1}) implies that the series (\ref{eq:BerensteinGay2}) converges uniformly for $|z|,|w| \le r$, as required. In particular, we have shown that the restriction map
  $$ \mathcal{R}_h : X_{\lambda}(\D^2) \to X_{\lambda}(\D) \, , \qquad \mathcal{R}_h(F)(z):=F(z,\overline{z}) \quad (z \in \D) \, $$
is bijective. Since $X_{\lambda}(\D^2)$ and $X_\lambda(\D)$ are both Fr\'echet spaces with respect to the topology of locally uniform convergence on $\D^2$ resp.~$\D$ (see \cite[Corollary 1 to Theorem 4.2.4]{Rudin2008} for the fact that $X_{\lambda}(\D)$ is a Fr\'echet space) and the restriction map $\mathcal{R}_h$ is obviously continuous, its inverse is continuous as well by the Open Mapping Theorem. This completes the proof of Theorem \ref{thm:DiskvsBidiskIntro}.
\end{proof}

  \begin{theorem}[Smooth eigenfunctions of $\Delta_{\widehat{\C}}$ on $\widehat{\C}$ vs.~holomorphic eigenfunctions of $\Delta_{zw}$ on $\Omega$] \label{thm:SpherevsBiPlane}
                  Let $\lambda \in \C $. Then the following are equivalent:
  \begin{itemize}
  \item[(i)] There is a function $f \in C^{2}(\widehat{\C})$ such that $\Delta_{\widehat{\C}} f=\lambda f$ on $\widehat{\C}$.
  \item[(ii)] There is a function $F \in \mathcal{H}(\Omega)$ such that $\Delta_{zw} F=\lambda F$ on $\Omega$.
  \item[(iii)] $\lambda=4 m (m+1)$ for some $m \in \N_0$.
    \end{itemize}
\end{theorem}
Theorem \ref{thm:SpherevsBiPlane} is a special case of Theorem \ref{thm:2Intro}.

		\begin{proof} The implication (ii) $\Rightarrow$ (iii) is Corollary \ref{cor:EW}, and (iii) $\Rightarrow$ (ii) is Theorem \ref{thm:helgason}. Clearly, (ii) implies (i), so we only need to prove that (i) implies (ii). Accordingly, we write $\lambda=4 \mu(\mu-1)$ with $\mu \in \C$ and $\Re\mu \ge 1/2$.
			For $n \in \Z$ consider
			$$ f_n(z):=\frac{1}{2\pi i} \int \limits_{\partial \D} f (\eta z) \, \eta^{-(n+1)} \, d \eta=\frac{1}{2\pi} \int \limits_{0}^{2\pi} f(e^{it} z) e^{-i n t} \, dt \, .$$
			Then $f_n \in C^{2}(\widehat{\C})$ is  $n$--homogeneous and
			$\Delta_{\widehat{\C}} f_n=4 \mu(\mu-1) f_n$ on $\widehat{\C}$. Arguing in a similar way as in the proof of Theorem \ref{thm:helgason} we see that there is
                        a constant $c_n \in \C$ such that
			$$ f_n(z)=c_n \binom{-\mu}{|n|} \Fz{\mu,1-\mu}{|n|+1}{\frac{|z|^2}{1+|z|^2}}  \cdot \begin{cases} \overline{z}^{|n|} & \text{ if } n \ge 0\,, \\            (-z)^{|n|} & \text{ if } n \le 0\, .\end{cases}$$Therefore, the behavior of $f_n$ as $z \to \infty$ depends essentially on the value $\prescript{}{2}{F_1}\left(\mu,1-\mu;|n|+1;1 \right)$. It is well--known, see \cite[Theorem 2.1.3]{askey}, that
			$$ \lim \limits_{x \to 1-} \frac{\Fz{\mu,1-\mu}{|n|+1}{1}}{-\log (1-x)}=\frac{1}{\Gamma(\mu) \Gamma(1-\mu)} \,  \quad
			\text{ if } n=0 \; ,  $$
			and, see \cite[Theorem 2.2.2]{askey},
			$$ \Fz{\mu,1-\mu}{|n|+1}{1} =\frac{\Gamma(|n|+1) \Gamma(|n|)}{\Gamma(|n|+1-\mu) \Gamma(|n|+\mu)} \,  \quad \text{ if } n\not=0 \, .$$
			Hence, if $n=0$, then $f_n$ is defined at the point $\infty$ only if $c_0=0$ or $\mu \in \N$.
			If $n\not=0$, then $\prescript{}{2}{F_1}\left(\mu,1-\mu,|n|+1,1 \right)
			=0 $ if and only if  $\mu=|n|+1+k$ for some $k \in \N_0$.
			Thus, $f_n :\widehat{\C} \to \C$ is defined at the point $\infty$ only if $c_n=0$ or $\mu \in \N$, $\mu > |n|$. We conclude that either $c_n=0$ for every $n \in \Z$ and then $f \equiv 0$, or $c_n\not=0$ for at least one $n \in \Z$ and then $\lambda=4 m(m+1)$ for some $m \in \N_0$. In the latter case, we see that $c_n\not=0$ forces $m+1=\mu=|n|+1+k$ for some $k \in \N_0$, so $m=|n|+k$. In particular, $c_n=0$ for all $|n|>m$, so $f(z)=F(z,-\overline{z})$ for
			\[ F(z,w)=\sum \limits_{n=-m}^m c_n P^{\mu}_n(z,w) \in \mathcal{H}(\Omega) \, .\qedhere \]
		\end{proof}
		\begin{remark} \label{rmk:Laplacian}
		We see that the spectrum of the hyperbolic Laplacian $\Delta_{\D}$ is $\C$ whereas the spectrum of the spherical Laplacian $\Delta_{\widehat{\C}}$ is notably smaller as it only consists of the scalars $4m(m+1)$ for $m\in\N_0$. Furthermore, by Theorem \ref{thm:SpherevsBiPlane} every eigenfunction of the spherical Laplacian can be extended to the whole domain $\Omega$. This is different to the hyperbolic case where the extension to $\D^2$ provided by Theorem \ref{thm:DiskvsBidiskIntro} is maximal at least for the category of rotationally invariant domains as the following example shows.
\end{remark}

              \begin{example}
			Fix $\mu\in\C$ with $\Re\mu\geq\frac{1}{2}$, let $(z_0,w_0)\in\Omega\setminus\D^2$ with $z_0 w_0 \not \in [1,\infty)$, and define
            \[b_n :=\left(\frac{\Gamma(\mu+|n|)}{\Gamma(\mu)|n|!}\prescript{}{2 }{F_1}\left(\mu,1-\mu;|n|+1;\frac{z_0w_0}{z_0w_0-1}\right)\right)^{-1}\, .\]
                        Note that asymptotically
                        $$|b_n| \sim \left( |n|+1 \right)^{1-\Re\mu} \qquad (|n| \to \infty) \, . $$
                        This follows from Lemma \ref{lem:2F1Asymptotic} and Stirling's formula (\ref{eq:Stirling}).
                        In view of (\ref{eq:PMEstimate}) this shows that the series $$F(z,w)=\sum\limits_{n=-\infty}^\infty b_n P_n^{\mu}(z,w)$$ converges locally uniformly in $\D^2$. Hence $F \in \mathcal{H}(\D^2)$ and, since every function $P^\mu_n$ is an eigenfunction of $\Delta_{zw}$, we have $\Delta_{zw} F=4 \mu(\mu-1) F$ in $\D^2$. In particular,
                        $$f(z):=F(z,\overline{z})=\sum \limits_{n=-\infty}^{\infty} b_n P^\mu_n(z,\overline{z})$$
                        is of class $C^{2}(\D)$ and an eigenfunction of $\Delta_{\D}$. By Theorem \ref{thm:DiskvsBidiskIntro}, $F \in \mathcal{H}(\D^2)$ is the uniquely determined function in $\mathcal{H}(\D^2)$ such that $F(z,\overline{z})=f(z)$ in $\D$.
Now assume $F\in\mathcal{H}(D)$, where~$D$ is a rotationally invariant domain such that $\D^2\subsetneq D \subseteq\Omega$ and $(z_0,w_0) \in D \setminus \D^2$. By Theorem \ref{thm:Helgason} there are coefficients $(\tilde{b}_n)_{n\in\Z}\subseteq \C$ such that $$F(z,w)=\sum\limits_{n=-\infty}^\infty \tilde{b}_n P_n^{\mu}(z,w)$$ for all $(z,w)\in D$. Since the coefficients are uniquely determined, we conclude $\tilde{b}_n=b_n$. This is~a contradiction because
			\[F(z_0,w_0)=\sum\limits_{n=0}^\infty w_0^n+\sum\limits_{n=1}^\infty z_0^n\]
			is a divergent series since $|z_0|\geq1$ or $|w_0|\geq1$.
   \end{example}

		\begin{corollary}
			\label{lem:NoContinuation}
			Let $\lambda\in\C$ and  $D$ a rotationally invariant subdomain of $\Omega$ such that $\D^2\subsetneq D$. Then there exists a function $f \in X_\lambda(\D^2)$ that cannot be analytically continued to $D$.
		\end{corollary}

\section{Poisson Fourier modes and the M{\"o}bius group}\label{sec:Pullback}\

   Recall the M{\"o}bius group $\mathcal{M}(\D^2)$ of the bidisk from \eqref{eq:OmegaAutomorphismBidisk}. In order to give a characterization of the  M{\"o}bius invariant eigenspaces of $\Delta_{zw}$ (e.g. proving Theorem \ref{thm:RudinIntro}), it will turn out that all we need to understand are precompositions of PFM with automorphisms in $\mathcal{M}(\D^2)$, that is
  \[P^{\mu}_n \circ T \text{ with }\Re \mu \ge 1/2 \text{ and }T \in \mathcal{M}(\D^2)\,.\]
  Moreover, since $\mathcal{M}(\D^2)$ is closely related to the M{\"o}bius group $\mathcal{M}$ of $\Omega$ from \eqref{eq:OmegaAutomorphismIntro}, we can use the following description of $\mathcal{M}$.
  \begin{lemma}
			\label{lem:MoebiusGenerators}
			The group $\mathcal{M}$ is generated by the flip map $\mathcal{F}(u,v) \coloneqq (1/v,1/u)$ and the mappings
			\begin{equation}
				\label{eq:MoebiusGenerators}
				T_{z,w}(u,v)
				\coloneqq
				\left(
				\frac{z-u}{1-wu} ,
				\frac{w-v}{1-zv}
				\right)
				\quad \textrm{and} \quad
				\varrho_{\gamma}(u,v)
				\coloneqq
				\left(
				\gamma u, \frac{1}{\gamma }v
				\right)
				\, , \qquad
				(u,v) \in \Omega
			\end{equation}
			with $(z,w) \in \Omega \cap \C^2$ and $\gamma \in \C^*$. More precisely, for every $T \in \mathcal{M}$ there exist $z,w \in \C$ and $\gamma \in \C^*$ such that
			\begin{equation*}
			\label{eq:MoebiusGeneratorsExplicit}
				T
				=
				\varrho_\gamma
				\circ
				T_{z,w}\quad\text{or}\quad T
				=
				\varrho_\gamma
				\circ
				T_{z,w}
				\circ
				\mathcal{F}
				\, .
			\end{equation*}
		\end{lemma}
  \begin{remark}
      Lemma \ref{lem:MoebiusGenerators} is exactly Lemma 2.2 in \cite{HeinsMouchaRoth2} where we have replaced the automorphisms $\Phi_{z,w}$ by the $T_{z,w} = \Phi_{z,w} \circ \varrho_{-1}$ automorphisms in \eqref{eq:MoebiusGeneratorsExplicit}. The reason will become apparent in \eqref{eq:MoebiusPoissonID} and \eqref{eq:PoissonkernelasMoebiustrafo}. Essentially, the automorphisms $T_{z,w}$ in \eqref{eq:MoebiusGenerators} are precisely the automorphisms of $\Omega$ interchanging a given point $(z,w)\in\Omega\cap(\C\times\C)$ with $(0,0)$ -- instead of only sending $(0,0)$ to $(z,w)$. See \cite[Sec.~2.3]{Annika} for more details on this.
  \end{remark}
		Therefore, understanding precompositions of a PFM $P_n^{\mu}$ with elements ${T\in\mathcal{M}(\D^2)}$ breaks down to understanding precompositions with the above generators. Note that the mappings $\mathcal{F}$ and $T_{z,w}$ for $w\neq\cc{z}$ are not elements of $\mathcal{M}(\D^2)$. However, the precompositions with these mappings still make sense for those PFM defined on all of $\Omega$, that is, in view of Theorem \ref{thm:helgason} the precompositions
        \[P^{-m}_n \circ T \text{ with }m\in\N_0\text{ and }T \in \mathcal{M}\]
        are well--defined.
        In the case that $T=\varrho_\gamma$ we have ${P_n^{\mu}\circ T=\gamma^{-n}P_n^{\mu}}$ by the $(-n)$--homogeneity of the PFM for all $\mu\in\C$. Next, if $T=\mathcal{F}$, it is easily seen by direct verification based on (\ref{eq:PoissonFourierExplicit}) that
        \begin{equation*}\label{eq:PoissonFourierforinversepoints}
				(P_{ n}^{-m}\circ\mathcal{F})(z,w)=(-1)^mP_{n}^{-m}(z,w)\,
			\end{equation*}
   for $m\in\N_0$ and $|n|\leq m$. In order to treat the case that $T=T_{z,w}$ resp.\ $T=T_{z,\cc{z}}$, some preliminary observations are useful. First, every automorphism $T_{z,w}$ is induced by a M{\"o}bius transformation $\psi_{z,w}$ of the form
		\begin{equation}\label{eq:phizw}
			\psi_{z,w}(u):=\frac{z-u}{1-wu}\,,\qquad
		\end{equation}
        in the sense that $T_{z,w}(u,v) = (\psi_{z,w}(u), 1/\psi_{z,w}(1/v))$.
		\begin{remark}
			Choosing $w=\overline{z}$ in \eqref{eq:phizw} yields all self--inverse automorphisms of $\D$ except for the identity, and, similarly, the choice {$w=-\overline{z}\in\D$} yields all self--inverse rigid motions of $\widehat{\C}$.
		\end{remark}
    Further, we will employ the definition of the PFM via integrals from \eqref{eq:PoissonFourier}. By Remark \ref{rem:whenPFMindeedisPoissonkernelpowerFouriermode}, when considering points $(z,\cc{z})\in\D^2$, this definition coincides with the $n$--th Fourier mode of the $\mu$--power of the Poisson kernel on the unit disk. The Poisson kernel on $\D$ satisfies the well--known properties
    \begin{equation}\label{eq:MoebiusPoissonID}
        P\left(z,\cc{z};\xi\right)=\frac{\psi_{z,\cc{z}}'(\xi)}{\psi_{z,\cc{z}}(\xi)}\xi
    \end{equation}
      and
    \begin{equation}\label{eq:PoissonkernelasMoebiustrafo}
					P\left(\psi_{z,\cc{z}}(u), \cc{\psi_{z,\cc{z}}(u)}; \psi_{z,\cc{z}}(\xi)\right)
					=
					\frac{P(u,\cc{u}; \xi)}{P(z,\cc{z};\xi)}
	\end{equation}
    where $z,u\in\D$, $\xi \in \del \D$ and $\psi_{z,\cc{z}}\in\mathcal{M}(\D)$, see \cite[Th.~3.3.5]{Rudin2008}.
    \begin{remark} \label{rem:DefP}
        One can define $P$ on $\Omega \times \partial\D$. Then $P$ is never zero, and for each $\xi \in \partial \D$ the function $P(\cdot;\xi)$ is meromorphic in the sense of \cite[Chapter VI.2]{FischerLieb2012}. The identities \eqref{eq:MoebiusPoissonID} and \eqref{eq:PoissonkernelasMoebiustrafo} then take the form
			\begin{align*}
				P\left(\psi_{z,w}^{-1}(0),1/\psi_{z,w}^{-1}(\infty);\xi\right)&=\frac{\psi_{z,w}'(\xi)}{\psi_{z,w}(\xi)}\xi =P\left(z,w;\xi\right)\, ,\\
				P\left(\psi_{z,w}(u),\frac{1}{\psi_{z,w}\left(1/v\right)};\psi_{z,w}(\xi)\right)&=\frac{P(u,v;\xi)}{P\left(\psi_{z,w}^{-1}(0),1/\psi_{z,w}^{-1}(\infty);\xi\right)}=\frac{P(u,v;\xi)}{P\left(z,w;\xi\right)} \,,
			\end{align*}
        where $(z,w)\in\Omega\cap\C^2$, $(u,v)\in\Omega$ and $\xi\in\partial\D$.
    \end{remark}

  With these preparations we are now in a position to analyze the precompositions of $P^{\mu}_n$ with elements of $\mathcal{M}(\D^2)$ resp.\ $\mathcal{M}$. We start with the simplest case, $P^{m+1}_0 \circ T_{u,v}$:
			\begin{proposition}\label{prop:Pullbackproposition}
				Let $m\in\N_0$, $(z,w)\in\Omega$ and $(u,v) \in \Omega \cap \C^2$. Then
				\begin{equation}\label{eq:MoebiusPullback}
					\big(
					P_0^{m+1} \circ T_{u,v}
					\big)(z,w)=\big(
					P_0^{-m} \circ T_{u,v}
					\big)(z,w)=\sum_{j=-m}^{m}
					P^{m+1}_{-j}\left(u,v\right)
					P^{-m}_{j}(z,w)\,.
				\end{equation}
			\end{proposition}

   \begin{proof}
				The way of reasoning is as follows: first, fix $z, u\in\D$ and let $v:=\overline{u}$. Then $\psi_{u,\cc{u}}\in\mathcal{M}(\D)$ which means, in particular, that $T_{u,\cc{u}}(z,\cc{z})\in\{(t,\cc{t})\,:\,t\in\D\}$. By the two variable identity principle, Lemma \ref{lem:RangeIDprincpiple}, if we can show \eqref{eq:MoebiusPullback} in this special case, then, keeping $u\in\D$ fixed, \eqref{eq:MoebiusPullback} also holds for points $(z,w)\in\Omega$ since both sides of the equation are holomorphic functions in $\Omega$. Next, fix $(z,w)\in\Omega$ and note that both sides of \eqref{eq:MoebiusPullback} are holomorphic as functions of $(u,v)\in\Omega \cap\C^2$. Assuming \eqref{eq:MoebiusPullback} for $v=\cc{u}\in\D$ then implies the claim, again by Lemma \ref{lem:RangeIDprincpiple}.

				It remains to show \eqref{eq:MoebiusPullback} for $w=\cc{z}$ and $v=\cc{u}$ with $z,u\in\D$. For this purpose we compute
    \begin{align*}
					\big(
					P_0^{m+1} \circ T_{u,\cc{u}}
					\big)(z,\cc{z})&\overset{\eqref{eq:Poissonplusminus}}{=}\big(
					P_0^{-m} \circ T_{u,\cc{u}}
					\big)(z,\cc{z})
					=
					\frac{1}{2\pi}
					\int\limits_{0}^{2\pi}
					P\left(\psi_{u,\cc{u}}(z),\frac{1}{\psi_{u,\cc{u}}\left(1/\cc{z}\right)};e^{it}\right)^{-m}\,dt\\
     &\overset{\eqref{eq:PoissonkernelasMoebiustrafo}}{=}
					\frac{1}{2\pi}
					\int\limits_{0}^{2\pi}
					P\big(z,\cc{z}; \psi_{u,\cc{u}}^{-1}(e^{it})\big)^{-m}
					P\big(u,\cc{u};\psi_{u,\cc{u}}^{-1}(e^{it})\big)^{m}\,dt\\
					&\overset{\psi_{u,\cc{u}}(e^{is})=e^{it}}{=}\frac{1}{2\pi}
					\int\limits_{0}^{2\pi}
					P\big(z,\cc{z}; e^{is}\big)^{-m}
					P\big(u,\cc{u};e^{is}\big)^{m+1}\,ds\,.
				\end{align*}
				In the last step we used \eqref{eq:MoebiusPoissonID}. Now we can interpret the resulting integral as the inner product on $L^2([0,2\pi],\C)$, the space of square integrable functions $f:[0,2\pi]\to\C$. Using Parseval's identity in (P) we obtain
				\begin{align*}
					\big(P_0^{-m} \circ T_{u,\cc{u}}\big)(z,\cc{z})&=\left\langle P\left(u,\cc{u};e^{is}\right)^{m+1}, \cc{P(z,\cc{z};e^{is})^{-m}}\right\rangle_{L^2([0,2\pi],\C)}\\
					&\overset{\mathrm{(P)}}{=}\sum_{j=-\infty}^{\infty}\left\langle	P\left(u,\cc{u};e^{it}\right)^{m+1},e^{-ijt}\right\rangle_{L^2([0,2\pi],\C)}
					\left\langle e^{-ijs},\cc{P(z,\cc{z};e^{is})^{-m}}\right\rangle_{L^2([0,2\pi],\C)}\\
					&=\sum_{j=-\infty}^{\infty}\frac{1}{(2\pi)^2}\int\limits_0^{2\pi}
					P\left(u,\cc{u};e^{it}\right)^{m+1}e^{ijt}\,dt
					\int\limits_0^{2\pi}P(z,\cc{z};e^{is})^{-m}e^{-ijs}\,ds\\
					&=\sum_{j=-m}^{m}P_{-j}^{m+1}(u,\cc{u})P_j^{-m}(z,\cc{z})
                    \, ,
				\end{align*}
				where the series reduces to a finite sum because $P_j^{-m}$ vanishes for $|j|>m$, see Remark~\ref{rem:PFMproperties}(b).
			\end{proof}

   The previous proof can be modified to establish the following result.

   \begin{theorem}\label{thm:Pullbackproposition}
				     Let $\mu\in\C$, $\Re\mu\geq1/2$, $n\in\N_0$ and $u\in\D$. Then
                \begin{subequations}\label{eq:invariance}
					\begin{align} \label{eq:Y2invariance}
						\big(
						P_{n}^{\mu} \circ T_{u,\cc{u}}
						\big)(z,w)&= \sum \limits_{j=-\infty}^{\infty} \left( \sum_{k=0}^n
						\sum_{\ell=0}^\infty
						(-1)^{k-n-\ell}
						\binom{n}{k}
						\binom{-n}{\ell}
						u^\ell
						\cc{u}^k
				P^{1-\mu}_{j-k+n+\ell}(u,\cc{u})\right)
						P^{\mu}_{-j}(z,w)\\ \label{eq:Y1inavariance}
						\big(
						P_{-n}^{\mu} \circ T_{u,\cc{u}}
						\big)(z,w)&=\sum\limits_{j=-\infty}^{\infty}\left(\sum_{k=0}^n
						\sum_{\ell=0}^\infty
						(-1)^{n-k+\ell}
						\binom{n}{k}
						\binom{-n}{\ell}
						u^k
						\cc{u}^{\ell}
						P^{1-\mu}_{j+k-n-\ell}(u,\cc{u})
						\right)P^{\mu}_{-j}(z,w) \,
					\end{align}
				\end{subequations}
    for all $(z,w) \in \D^2$.
                Moreover, both series converge locally uniformly and absolutely w.r.t.\ $(z,w)$ in $\D^2$.
			\end{theorem}
			\begin{proof}
    Using \eqref{eq:MoebiusPoissonID} and \eqref{eq:PoissonkernelasMoebiustrafo} we compute
				\begin{align*}
					\big(
					P_n^{\mu} \circ T_{u,\cc{u}}
					\big)(z,\cc{z})
					&=\frac{1}{2\pi}
					\int\limits_0^{2\pi}
					P\big(z,\cc{z}; e^{is}\big)^{\mu}
					P\big(u,\cc{u};e^{is}\big)^{1-\mu}\psi_{u,\cc{u}}(e^{is})^{-n}\,ds\\
					&=
                    \nonumber
                    \frac{1}{2\pi }	\int\limits_0^{2\pi}
					P\big(z,\cc{z}; e^{is}\big)^{\mu}
					P\big(u,\cc{u};e^{is}\big)^{1-\mu}(1-\cc{u}e^{is})^n(u-e^{is})^{-n}\,ds\,.
				\end{align*}
				Note that taking complex powers is unproblematic since the appearing Poisson kernels are positive real quantities. Applying Lemma \ref{lem:RangeIDprincpiple} and, additionally, the generalized binomial theorem leads to
                \begin{align*}
					\big(
					P_n^{\mu} &\circ T_{u,\cc{u}}
					\big)(z,w)\nonumber\\
					&=
					\sum_{k=0}^n
					\sum_{\ell=0}^\infty
					\binom{n}{k}
					\binom{-n}{\ell}
					(-1)^{k-n+\ell}\cc{u}^ku^\ell
					\frac{1}{2\pi }
					\int\limits_0^{2\pi}
					P\big(z,w; e^{is}\big)^{\mu}
					P\big(u,\cc{u};e^{is}\big)^{1-\mu}e^{is(k-n-\ell)}\, ds\,.
				\end{align*}
                Interpreting the above integral as a $L^2([0,2\pi],\C)$ inner product as it was done in the proof of Proposition \ref{prop:Pullbackproposition} shows \eqref{eq:Y2invariance}. A similar computation yields \eqref{eq:Y1inavariance}.

                Let $K\subseteq\D^2$ compact. It remains to prove that
                \begin{align*}
				\sum_{k=0}^n\sum_{\ell=0}^\infty \left\vert\binom{n}{k}\binom{-n}{\ell}\overline{u}^ku^\ell\right\vert
				\sum_{j=-\infty}^{\infty}
						\left\vert P^{1-\mu}_{j-k+n+\ell}(u,\cc{u})\right\vert
						\sup_{(z,w)\in K}\left\vert P^{\mu}_{-j}(z,w)\right\vert<\infty\,.
			\end{align*}
   We have
   \begin{align*}
       \left\vert P^{1-\mu}_{j-k+n+\ell}(u,\cc{u})\right\vert\leq\frac{1}{2\pi}\int\limits_0^{2\pi}\left\vert P(u,\cc{u};e^{it})^{1-\mu}\right\vert\, dt:=M_u<\infty
   \end{align*}
   as $u\in\D$ is fixed. Since $K$ is compact, we find a non--negative real number $r<1$ such that $|z|,|w|\leq r$ for all $(z,w)\in K$. Hence, using \eqref{eq:PMEstimate} we can ensure that
   \begin{align*}
       \sum_{j=-\infty}^{\infty}\sup_{(z,w)\in K}\left\vert P^{\mu}_{-j}(z,w)\right\vert&\leq\sum_{j=-\infty}^{\infty}\gamma_r \cdot \left(|j|+1 \right)^{\Re\mu-1} \, r^{|j|}:=N_{r,\mu}<\infty\,,
   \end{align*}
   where $\gamma_r>0$ depends on $r$ (and $\mu$). Thus,
   \begin{align*}
				\sum_{k=0}^n\sum_{\ell=0}^\infty &\left\vert\binom{n}{k}\binom{-n}{\ell}\overline{u}^ku^\ell\right\vert
				\sum_{j=-\infty}^{\infty}
						\left\vert P^{1-\mu}_{j-k+n+\ell}(u,\cc{u})\right\vert
						\sup_{(z,w)\in K}\left\vert P^{\mu}_{-j}(z,w)\right\vert\\ &\leq\sum_{k=0}^n\sum_{\ell=0}^\infty \left\vert\binom{n}{k}\binom{-n}{\ell}\overline{u}^ku^\ell\right\vert
				\cdot M_u\cdot N_{r,\mu}<\infty
			\end{align*}
   since the binomial series over $k$ and $\ell$ converge absolutely.
		\end{proof}

        Let $m,n\in\N_0$, $n\leq m$. For $\mu=-m$ the $j$--series in \eqref{eq:invariance} terminate. Moreover in this case, when replacing $(u,\cc{u})$ by $(u,v)\in\D^2$ in \eqref{eq:invariance}, a similar argument as in the proof of Theorem \ref{thm:Pullbackproposition} shows that the series in \eqref{eq:invariance} are absolutely and locally uniformly convergent with respect to $(u,v)$ in $\D^2$, too.  Further, we know that $P_{\pm n}^{-m}\circ T_{u,v}\in\mathcal{H}(\Omega)$. Thus, both sides of \eqref{eq:Y2invariance} resp.\ \eqref{eq:Y1inavariance} (with $(u,\overline{u})$ replaced by $(u,v)$) define holomorphic functions w.r.t.\ $(u,v)\in\D^2$ which allows us to apply the two variable identity principle, Lemma \ref{lem:RangeIDprincpiple}. We obtain:
        \begin{corollary}\label{rem:Pullbackproposition}
            Let $m,n\in\N_0$, $n\leq m$, $(z,w)\in\Omega$ and $(u,v)\in\D^2$. Then
            \begin{subequations}
                \begin{align*}
      \big(
						P_{n}^{-m}&\circ T_{u,v}
						\big)(z,w)
						=\sum_{j=-m}^m \left( \sum_{k=0}^n
						\sum_{\ell=0}^\infty
						(-1)^{k-n-\ell}
						\binom{n}{k}
						\binom{-n}{\ell}
						u^\ell
						v^k
						P^{m+1}_{j-k+n+\ell}(u,v)
						\right) P^{-m}_{-j}(z,w)\\
      \big(
						P_{-n}^{-m}&\circ T_{u,v}
						\big)(z,w)
						=\sum\limits_{j=-m}^m \left( \sum_{k=0}^n
						\sum_{\ell=0}^\infty
						(-1)^{n-k+\ell}
						\binom{n}{k}
						\binom{-n}{\ell}
						u^k
						v^{\ell}
						P^{m+1}_{j+k-n-\ell}(u,\cc{u})
						\right)P^{-m}_{-j}(z,w)\, .
					\end{align*}
            \end{subequations}
    In particular, $P^{-m}_n \circ T_{u,v} \in \mathrm{span}\{P^{-m}_{j} \, : \, j=-m, \ldots, m\} \subseteq\mathcal{H}(\Omega)$.
        \end{corollary}

\section{Proof of Theorems \ref{thm:2Intro} and \ref{thm:RudinIntro}:  invariant eigenspaces}\label{sec:PFandRudin}\

		Using the theory we have developed so far, we are now able to give an explicit description of the eigenspaces of the Laplace operator $\Delta_{zw}$ on $\D^2$ resp.\ on $\Omega$. This will allow us to conclude Theorems \ref{thm:2Intro} and \ref{thm:RudinIntro} afterwards.

		\begin{theorem}[M{\"o}bius invariant subspaces of $X_\lambda(\D^2)$]\label{thm:RudinthmforOmega}
			Let $\lambda\in\C$.
			\begin{itemize}
				\item[(NE)] If $\lambda\neq4m(m+1)$ for all $m\in\N_0$, then $X_\lambda(\D^2)$ has no non--trivial M{\"o}bius invariant subspaces.
				\item[(E)] If $\lambda=4m(m+1)$ for some $m\in\N_0$, then $X_\lambda(\D^2)$ has precisely three non--trivial M{\"o}bius invariant subspaces $X_{\lambda}^{+}(\D^2),X_{\lambda}^{-}(\D^2),X_{\lambda}^{0}(\D^2)$ of which exactly one, say $X_{\lambda}^{0}(\D^2)$, is finite dimensional of dimension $2m+1$. Explicitly, these spaces are given by
    \begin{subequations}\label{eq:Rudinexceptional}
					\begin{align}\label{eq:RudinPlus}
						X_{\lambda}^{+}(\D^2)&=\clos{\D^2} \left(\mathrm{span}\{P_n^{m+1} \, : \, -m\leq n< \infty\}\right)
                                   \,,\\
						\label{eq:RudinMinus}
						X_{\lambda}^{-}(\D^2)&=\clos{\D^2}\left(\mathrm{span}\{P_n^{m+1} \, : \, -\infty< n\leq m\}\right)
      \,,\\
						\label{eq:RudinFinite}
						X_{\lambda}^{0}(\D^2)&=\mathrm{span}\{P_n^{m+1} \, : \, -m\leq n\leq m\} \, ,
					\end{align}
				\end{subequations}
				where $\clos{\D^2}$ denotes the closure with respect to the topology of locally uniform convergence on $\D^2$.
			\end{itemize}
		\end{theorem}
		Note that Rudin gives an explicit characterization of the three non--trivial subspaces in the second case of his theorem (see Theorem \ref{thm:IntroRudin}), too. We will discuss the similarities and the differences between Rudin's and our approach in Remark~\ref{rem:RudintheoremOmega}.

	Three crucial ingredients that we will use for the proof of Theorem \ref{thm:RudinthmforOmega} are the representation of eigenfunctions in terms of PFM from Theorem~\ref{thm:Helgason}, the pullback formula for PFM from Theorem~\ref{thm:Pullbackproposition} and the differential operators $D^{+}$ and $D^{-}$ defined by
		\begin{equation}
        \label{eq:dualPM}
			D^+ f(z,w)
            \coloneqq
            \partial_z f(z,w)-w^2 \partial_w f(z,w) \, , \qquad
            D^- f(z,w)
            \coloneqq
            \partial_w f(z,w)-z^2 \partial_z f(z,w)
		\end{equation}
		for $f \in \mathcal{H}(D)$ and any subdomain $D \subseteq \C^2$. We note an alternative description of theses operators based on the automorphisms from \eqref{eq:MoebiusGenerators}.
        \begin{lemma}
            Let $D \subseteq \C^2$ be a subdomain,  $f \in \mathcal{H}(D)$. For all $(z,w) \in D$ it then holds that
            \begin{equation}
            \begin{split}
                \label{eq:firstdualPMalternativedefinition}
                D^{+} f(z,w)
                =
                \partial_u
                \left(
                    f
                    \circ
                    T_{u,0}
                    \circ {\varrho}_{-1}
                \right)
                \Big|_{u=0}
                (z,w) \, , \\
                D^{-} f(z,w)
                =
                \partial_v
                \left(
                    f
                    \circ
                    T_{0,v} \circ \varrho_{-1}
                \right)
                \Big|_{v=0}
                (z,w) \, .
            \end{split}
            \end{equation}
        \end{lemma}

        What makes these operators useful is that they act by shifts of the Fourier index on the PFM.
		\begin{lemma}\label{lem:PoissonFouriershiftbyDualPM}
			The operators $D^+$ and $D^-$ commute with $\Delta_{zw}$. Furthermore, for $\mu\in\C$ and $n\in\Z$ it holds that
			\begin{equation*}
				\label{eq:DualPMShift}
				D^+
                P_n^{\mu}
                =
                (\mu-n-1)
                P_{n+1}^{\mu}
                \qquad\text{and}\qquad
                D^-
                P_n^{\mu}
                =
                (\mu+n+1)
                P_{n-1}^{\mu} \, .
			\end{equation*}
		\end{lemma}

		\begin{proof}
			The operator $D^+$ commutes with $\Delta_{zw}$ in view of the latter's $\mathcal{M}$--invariance, \eqref{eq:firstdualPMalternativedefinition} and commutativity of partial derivatives. Moreover, $D^+$ lowers homogeneity in the sense of \eqref{eq:3.1} by one degree. Since $P_n^{\mu}$ is a $(-n)$--homogeneous eigenfunction of $\Delta_{zw}$ to the eigenvalue ${4\mu(\mu-1)}$, this implies that $D^+ P_n^{\mu}$ is a ${-(n+1)}$--homogeneous eigenfunction of $\Delta_{zw}$ to the same eigenvalue. Hence, Theorem~\ref{thm:helgason} yields ${D^+ P_n^{\mu} = c P_{n+1}^{\mu}}$ for some ${c \in \C}$. Applying \eqref{eq:dualPM} and evaluating \eqref{eq:generalPoissonfouriermode} in $(z,w)=(0,1)$ yields $c = \mu - n - 1$. Note that one has to distinguish the cases $n \ge 0$ and $n < 0$. The second equality in \eqref{eq:DualPMShift} may be derived analogously.
		\end{proof}
        \begin{remark}
            One may thus regard the operators $D^+$ and $D^-$ as ladder operators and together with their commutator
        \begin{equation*}
            \frac{1}{2}
            \big(
                D^- \circ D^+
                -
                D^+ \circ D^-
            \big)
            =
            z \partial_z
            -
            w \partial_w
            \, ,
        \end{equation*}
        which is the Euler vector field corresponding to the homogeneity \eqref{eq:3.1}, they generate the Lie algebra of $\mathcal{M}$. By \eqref{eq:firstdualPMalternativedefinition}, all three operators may moreover be regarded as fundamental vector fields of the natural action of $\mathcal{M}$ as automorphisms of $\Omega$.
        \end{remark}

		\begin{proof}[Proof of Theorem \ref{thm:RudinthmforOmega}]\
  Write $\lambda=4\mu(\mu-1)$ for $\mu\in\C$, $\Re\mu\geq1/2$, and let $Y$ be a non--trivial M{\"o}bius invariant subspace of $X_\lambda(\D^2)$. The way of reasoning is as follows: we first show that $Y$ contains a PFM $P_n^{\mu}$ for some $n\in\Z$. If the assumption of (NE) holds, then this will imply that $P_n^{\mu}\in Y$ for all $n\in\Z$. Conversely, if the assumption of (E) holds for $m\in\N_0$, the existence of $P_n^{-m}$ for some $n\in\Z$ will imply that certain other PFM need to be contained in $Y$, too. Here, we will need to distinguish three cases which will lead to the three specific spaces described on the right--hand side of \eqref{eq:Rudinexceptional}. This shows that there are at most three possible choices for $Y$. Then, in a second step, we show that the three spaces found in the first step are in fact M{\"o}bius invariant, which then proves the existence of exactly three non--trivial proper M{\"o}bius invariant subspaces.
			\begin{mylist}
				\item[Step 1:]
				\begin{itemize}
					\item[(i)] Using \eqref{eq:dualPM} it follows from $Y$ being closed and M{\"o}bius invariant that the differential quotients $D^+ f, D^- f\in Y$. Therefore, $Y$ is invariant with respect to $D^+$ and $D^-$, and, by iteration, $(D^+)^kf,(D^-)^kf\in Y$ for every  $k\in\N$.

					\item[(ii)] Since $Y$ is M{\"o}bius invariant, $Y$ is, in particular, rotationally invariant. Thus, for every $n\in\Z$, we have $f_n\in Y$ where $f_n$ is defined by
					$$ f_n(z,w):=\frac{1}{2\pi i} \int \limits_{\partial \D} f \left( \eta z,\frac{w}{\eta} \right) \eta^{-(n+1)} \, d\eta \, .$$
					By Theorem \ref{thm:helgason}, $f_n$ is a multiple of the PFM $P_{n}^{\mu}$. Moreover, by Theorem \ref{thm:Helgason}, there are $c_n \in \C$ for all $n \in \Z$ depending on $f$ such that $f$ has the representation
					$$f=\sum_{n=-\infty}^\infty c_nP_n^{\mu} \, .$$
					Since $Y\neq\{0\}$, we may assume $f\not\equiv0$, i.e. $c_n \neq 0$ for some $n \in \Z$. Thus, the above observations imply $P_n^{\mu}\in Y$ for some $n\in\Z$.

					\item[(iii)] By Part (ii) there exists $P_n^{\mu}\in Y$ for some $n \in \Z$, and by Part (i) we conclude ${(D^+)^kP_n^{\mu},(D^-)^kP_n^{\mu}\in Y}$ for every $k\in\N$. Lemma \ref{lem:PoissonFouriershiftbyDualPM} shows that these functions are multiples of PFM again, so ${(\mu -n-k)_k P_{n+k}^{\mu }}$, $ (\mu +n-k)_k P_{n-k}^{\mu}\in Y$ for every $k\in\N$. Recall that the prefactors denote (rising) Pochhammer symbols, see \eqref{eq:Pochhammer}.

					For $\mu\not\in\Z$ the factors $(\mu \pm n-k)_k$ never vanish. In this case, $P_n^{\mu}\in Y$ for every $n\in\Z$. Using Theorem \ref{thm:Helgason} again, this implies $Y=X_\lambda(\D^2)$ which proves (NE).

					For $\mu \in\N$ the factors $(\mu \pm n-k)_k$ vanish for appropriate $k \in \N$. This leads to the dichotomy in case (E): assume $\mu = m+1\in\N$, and $P_n^{m+1}\in Y$. Then necessarily $P_{k}^{m+1}\in Y$ for $-m\leq k\leq m$. Since shifting $P_{k}^{m+1}$ with $D^\pm$ eventually produces the zero function, $P_{q}^{m+1}$, $q<-m$ or $q>m$, does not necessarily need to be contained in $Y$. However, if $P_{q}^{m+1}\in Y$ for some $q<-m$, then necessarily all $P_{p}^{m+1}$, $-\infty<p\leq m$, need to be contained in $Y$ since $P_{q}^{m+1}$ can be shifted to each of these function without producing the zero function. The same argument works for $q>m$. In summary, these considerations show that every non--trivial closed M{\"o}bius invariant subspace of $X_{\lambda}(\D^2)$ contains one of the three spaces
                    \begin{align*}
                         Y_{+}&:=\clos{\D^2}\left(\mathrm{span}\{P_n^{m+1} \, : \, -m\leq n< \infty\}\right)
                         \,,\\
                         Y_{-}&:=\clos{\D^2}\left(\mathrm{span}\{P_n^{m+1} \, : \, -\infty< n\leq m\}\right)
                         \, \text{ and}\\
                        Y_{0}\,&:=\mathrm{span}\{P_n^{m+1} \, : \, -m\leq n\leq m\}\,.
                     \end{align*}
                    \end{itemize}
				\item[Step 2:] It remains to show that $ Y_{+},Y_-$ and $Y_0$ are in fact M{\"o}bius invariant. Let $\mu-1=m\in\N_0$, $n\in\Z$ and $z,w,u\in\D$.
				\begin{itemize}
					\item[(iv)] Assume $n\geq0$. Theorem \ref{thm:Pullbackproposition} shows that $(P_{n}^{m+1} \circ T_{u,\cc{u}})(z,w)$ is a series of scalar multiples of $P_{k}^{m+1}(z,w)$ with $-m\leq k<\infty$. To see this recall that by Remark \ref{rem:PFMproperties} we have $P_n^{-m}\equiv0$ if $|n|>m$. Moreover, Theorem \ref{thm:Pullbackproposition} also implies that this series converges locally uniformly with respect to $(z,w)$ on $\D^2$. Thus, $P_{n}^{m+1} \circ T_{u,\cc{u}} \in Y_{+}$. The same argument applies to $P_{-n}^{m+1} \circ T_{u,\cc{u}}$.

                    \item[(v)] Part (iv) implies that every linear combination
					$$\sum_{n=-m}^N c_n P_n^{m+1}\circ T_{u,\cc{u}}  \, $$
     belongs to $Y_{+}$. Since, by definition, every element in $Y_{+}$ is a series of the form $\sum_{n=-m}^{\infty} c_n P^{m+1}_n$, M{\"o}bius invariance of $Y_{+}$ follows. The same argument applies to~$Y_{-}$.

					\item[(vi)] Now, consider $-m\leq n\leq m$. In this case $P_{ n}^{m+1}$ is a multiple of $P_{n}^{-m}$ by \eqref{eq:Poissonplusminus}, and Corollary \ref{rem:Pullbackproposition} shows that $(P_{n}^{m+1}\circ T_{u,\cc{u}})(z,w)$ can be expressed as linear combination of $P_{k}^{-m}(z,w)$  with $-m\leq k\leq m$. This shows M{\"o}bius invariance of $Y_{0}$.\qedhere
				\end{itemize}
			\end{mylist}
		\end{proof}

  \begin{remark}\label{rem:RudintheoremOmega}
			\begin{itemize}
				\item[(a)] In his proof of Theorem \ref{thm:IntroRudin} Rudin employs exclusively a specific   differential operator $A$, see \cite[Formula (5), p.~143]{Rudin84}, which in our terminology can be defined by
				$$ A:=D^+ + D^-\,.$$
				The operator ${A^{\sharp}:=i \left(D^+ - D^- \right)}$ is  ``conjugate'' to $A$ in the sense that
				$$ A-i A^{\sharp}=2D^+ \quad \text{ and } \quad  A+i A^{\sharp}=2 D^- \, .$$
				Note the analogy with the Wirtinger derivatives $\partial_z$ and $\partial_{\bar{z}}$ which satisfy
				$$ \partial_x -i \partial_y=2 \partial_z \quad \text{ and } \quad \partial_x+i \partial_y=2 \partial_{\bar{z}} \, . $$
    We note that in a related, but different  context, namely in \cite{Rudin83}, Rudin also considers the operators $D^+$ and $D^-$. He denotes them by $Q$ and $\overline{Q}$, see \cite[Section 4.1, Formula (1)]{Rudin83}.

				\item[(b)] Step 1 of our proof closely follows the argumentation of Rudin's proof (Steps 1--3 in his labelling). However, because of the usage of his $A$--operator (see Part (a)), which does not respect homogeneity when applied to a PFM, Rudin additionally needs  to consider the projections on $n$--homogeneous components in his proof. In our approach based on the $D^{\pm}$--operators this issue does not arise, since these operators lower resp.\ raise the degree of homogeneity by exactly one.

				\item[(c)] Step 2 of our proof, where we made heavy use of Proposition \ref{prop:Pullbackproposition} resp.\ Theorem \ref{thm:Pullbackproposition}, is different to Rudin's approach (Step 4 in his labeling). In short, Rudin makes a series expansion in the disk automorphism parameter and uses again that M{\"o}bius invariant subspaces are invariant under repeated application of $A$ (see Part (a)).

				\item[(d)] The proof of Theorem \ref{thm:RudinthmforOmega} heavily relies on the PFM. In fact, the usage of $n$--homogeneous functions is also crucial in the original proof of Theorem \ref{thm:IntroRudin} in \cite{Rudin84}, although Rudin does not explicitly define these functions but works with projections on the subspaces of $n$--homogeneous eigenfunctions.
			\end{itemize}
		\end{remark}

		We are now in a position to prove Theorems \ref{thm:RudinIntro} and \ref{thm:2Intro}.

		\begin{proof}[Proof of Theorem~\ref{thm:RudinIntro}]
			Let $\lambda \in \C$ and $Y \subseteq \mathcal{H}(\D^2)$ be a non--trivial M{\"o}bius invariant subspace of $X_\lambda(\D^2)$. Let first $Y = X_\lambda(\D^2)$ be the full eigenspace. Every PFM is holomorphic at least on $\Omega_*$.  By rotational invariance of $\D^2$ and Theorem~\ref{thm:Helgason} this implies that $Y \cap \mathcal{H}(\Omega_*)$ is dense in $Y$, i.e. (NE) holds.  Assume now that $Y$ is a proper subspace of $X_\lambda(\D^2)$. By Theorem~\ref{thm:RudinthmforOmega}, this is only possible if $\lambda$ is an exceptional eigenvalue. Furthermore,
			$$ Y
			\in
			\big\{
			X_{\lambda}^{+}(\D^2),
			X_{\lambda}^{-}(\D^2),X_{\lambda}^{0}(\D^2)
			\big\}
			\, .$$
			Combining \eqref{eq:RudinPlus}, \eqref{eq:RudinMinus}, resp. \eqref{eq:RudinFinite} with Corollary \ref{cor:EW} yields ($\text{E}_+$\,), ($\text{E}_-$\,) resp. ($\text{E}_0$\,). We have already discussed the additional statements in the latter case in Theorem~\ref{thm:RudinthmforOmega}.

			Finally, Corollary \ref{lem:NoContinuation} shows that none of the density statements ($\text{E}_+$\,), ($\text{E}_-$\,) and (NE) may be improved to holomorphic extensibility  of every element.
		\end{proof}

        \begin{proof}[Proof of Theorem~\ref{thm:2Intro}]
            The equivalence of (i), (ii) and (iii) was already established in Theorem~\ref{thm:SpherevsBiPlane}. Let $\lambda = 4m(m+1)$ for some $m \in \N_0$. By \eqref{eq:RudinFinite} we then have $\dim(X_\lambda(\Omega)) = 2m+1$. Investing moreover Theorem~\ref{thm:helgason} yields (a). Remark~\ref{rmk:Laplacian} proves the inclusion ``$\subseteq$'' in \eqref{eq:SphericalEigenfunctions}. Conversely, restricting an eigenfunction $F \in X_\lambda(\Omega)$ to the rotated diagonal $\{(z,-\cc{z}) \colon z \in \widehat{\C}\}$ yields an eigenfunction of $\Delta_{\widehat{\C}}$. This completes the proof.
        \end{proof}

        The finite--dimensional spaces $X_{4m(m+1)}(\Omega)$ can also be characterized in terms of the PFM $P_0^{-m}$ only.
        \begin{corollary}\label{cor:cyclicPoissonFouriermode}
			Let $m\in\N_0$. The function $P_0^{-m}$ is cyclic in $X_{4m(m+1)}(\Omega)$ with respect to the natural action of $\mathcal{M}$ by pullbacks. That is,
			\begin{equation}\label{eq:cyclicPoissonFouriermode}
                X_{4m(m+1)}(\Omega)
                =
                \operatorname{span}
				\big\{
				P_0^{-m} \circ T
				\;\big|\;
				T \in \mathcal{M}
				\big\}
                \, .
			\end{equation}
    In addition,
    \begin{equation}\label{eq:cyclicPoissonFouriermodehypandsphere}
                X_{4m(m+1)}(\Omega) =
                \operatorname{span}
				\big\{
				P_0^{-m} \circ T_{u,\cc{u}}
				\;\big|\;
				u \in \D
				\big\}
                =
                \operatorname{span}
				\big\{
				P_0^{-m} \circ T_{u,-\cc{u}}
				\;\big|\;
				u \in \C
				\big\}
                 \, .
			\end{equation}
		\end{corollary}
  \begin{proof}
        We begin by showing the first equality in \eqref{eq:cyclicPoissonFouriermodehypandsphere}, which implies \eqref{eq:cyclicPoissonFouriermode}. Recall that, by \eqref{eq:RudinFinite}, the set $\{P^{-m}_j \colon -m \le j \le m\}$ constitutes a basis of $X_{4m(m+1)}(\Omega)$. By \eqref{eq:MoebiusPullback},
			\begin{equation*}
				P_0^{-m} \circ T_{u,\cc{u}}
				=
				\sum_{j=-m}^{m}
				P^{m+1}_{-j}(u,\cc{u})
				P^{-m}_{j}
			\end{equation*}
        for every $u \in \D$. We interpret the claim as a change of basis from ${\{P^{-m}_{j} \;|\; -m \le j \le m\}}$ to ${\{P_0^{-m} \circ T_{u_k,\cc{u}_k} \;|\; -m \le k \le m\}}$. Therefore, we have to find suitable $u_{-m}, \ldots, u_m \in \D$ such that the coefficient matrix
			\begin{equation*}
				\big( P_j^{m+1}(u_k,\cc{u}_k) \big)_{-m \le j,k \le m}
			\end{equation*}
		is invertible. However, it is true in general that given linearly independent complex valued functions $F_1,\dots, F_M$ for some $M\in\N$ on a set containing at least $M$ elements, it is possible to find the same number of points $x_1,\dots, x_M$ in the same set such that the matrix $(F_j(x_k))_{1\leq j,k\leq M}$ is invertible, see \cite[Proof of Prop.~7.28]{gallier2020}.  The second equality in \eqref{eq:cyclicPoissonFouriermodehypandsphere} follows similarly.
	   \end{proof}

		\section{Poisson Fourier modes and Peschl--Minda operators}
        \label{sec:PoissonFouriervsPeschlMinda}

       In \cite{HeinsMouchaRoth2}  the classical Peschl--Minda differential operators, which were introduced by Peschl \cite{Peschl1955} and studied e.g.~by \cite{Ahar69,Harmelin1982,KS07diff,KS11,Minda,Schippers2003,Schippers2007,Sugawa2000}, have been extended to differential operators acting on holomorphic functions defined on subdomains of $\Omega\cap \C^2$. It is the purpose of this section to put these (generalized) Peschl--Minda operators into the context of the present paper. In particular, we relate the Poisson Fourier modes with the Peschl--Minda operators.

        We first briefly recall  the definition from \cite{HeinsMouchaRoth2}.
                Let $U$ be an open subset of $\Omega \cap \C^2$ and ${f \in C^{\infty}(U)}$. The Peschl--Minda derivative $D^{m,n}f$ at the point $(z,w) \in U$ is defined by
		$$D^{m,n} f(z,w):=\frac{\partial^{m+n}}{\partial u^m \partial v^n} \left( f \circ {T}_{z,w} \circ \varrho_{-1} \right)(u,v) \bigg|_{(u,v)=(0,0)}  \, .$$

        We write $D_z^n \coloneqq D^{n,0}$ and $D_w^n \coloneqq D^{0,n}$ and refer to these operators as \emph{pure} Peschl--Minda operators. Comparing with \eqref{eq:firstdualPMalternativedefinition}, we note that the roles of $(z,w)$ and $(u,v)$ have been swapped, which yields the operators $D^1_z$ and $D^1_w$ instead of $D^+$ and $D^-$, respectively.

        \medskip

        Given a linear mapping $T \colon V \longrightarrow W$ we write $\ker(T) \coloneqq \{x \in V \colon Tx = 0\}$ for its kernel. We observe that the pure Peschl--Minda operators reproduce the generalized Poisson kernel from \eqref{eq:generalizedPoisson} in the following way:
		\begin{lemma}\label{l:PeschlMindaPoissonkernels}
			Let $\mu \in \C$, $(z,w)\in\D^2$, $\xi\in\partial\D$ and $k \in \N_0$. Then
   \begin{subequations}
			\begin{align}
				\label{eq:DzderivativePa}
                D^k_zP(z,w;\xi)^{\mu}&=(\mu)_{k}P(z,w;\xi)^{\mu}\left(\psi_{z,w}(\xi)\right)^{-k}\\
                D^k_wP(z,w;\xi)^{\mu}&=(\mu)_{k}P(z,w;\xi)^{\mu}\left(\psi_{z,w}(\xi)\right)^{k}
			\end{align}
   \end{subequations}
			with $\psi_{z,w}$ from \eqref{eq:phizw}. In particular, if $\mu=-m\in (-\N_0)$, then
			\begin{equation*}\label{eq:PoissonkernelinPMkernel}
				P^{-m}
				\in
				\ker(D^{m+1}_z) \cap \ker(D_w^{m+1})\,.
			\end{equation*}
		\end{lemma}
            \begin{proof}
                If $w=\cc{z}$, this is most easily proved by an induction on $k$ and the help of Proposition 3.8 in \cite{HeinsMouchaRoth2}. Then, the general result follows from Lemma \ref{lem:RangeIDprincpiple}.
            \end{proof}
             We note in passing that Lemma \ref{l:PeschlMindaPoissonkernels} for $w=\cc{z}$ says that up to multiplication with a unimodular constant the (classical) Poisson kernel $z \mapsto P(z,\overline{z};\xi)$ is a joint eigenfunction of the (classical) Peschl--Minda operators studied e.g.~in \cite{KS07diff}. Remarkably, the Peschl--Minda operators also act as weighted shifts when applied to the zeroth PFM.
		\begin{proposition}\label{prop:shift}
			For $\mu \in \C$ and $n \in \N_0$ we have
			\begin{equation*}
                \label{eq:shiftPM}
				(-\mu+1)_n
				\cdot
				P_n^{\mu}
				=
				D^n_z
				P^{\mu}_0
				\quad \textrm{and} \quad
				(-\mu+1)_n
				\cdot
				P_{-n}^\mu
				=
				D^n_w
				P^{\mu}_0 \, .
			\end{equation*}
		\end{proposition}
  \begin{proof}
		Let $z \in \D$. We prove the claim for points $(z,\cc{z})$ first from which then follows the result by the identity principle (Lemma \ref{lem:RangeIDprincpiple}). We compute
        \begin{align*}
				D^{n,0}P_0^{\mu}(z,\cc{z})&=\frac{ 1}{2\pi}
				\int\limits_{0}^{2\pi}
				D^{n,0}P\big(z,\cc{z};e^{it}\big)^{\mu}\, dt \\
				&\overset{\eqref{eq:DzderivativePa}}{=}\frac{1}{2\pi}
				\int\limits_{0}^{2\pi}(\mu)_n P\big(z,\cc{z};e^{it}\big)^{\mu}\left(\psi_{z,\cc{z}}(e^{it})\right)^{-n}\, dt\\
				&\overset{\eqref{eq:PoissonkernelasMoebiustrafo}}{=}\frac{(\mu)_n }{2\pi}
				\int\limits_{0}^{2\pi}P\big(z,\cc{z};\psi_{z,\cc{z}}(e^{it})\big)^{-\mu}\left(\psi_{z,\cc{z}}(e^{it})\right)^{-n}\, dt\\
				&\overset{\eqref{eq:MoebiusPoissonID}}{=}\frac{(\mu)_n }{2\pi}
				\int\limits_{0}^{2\pi}P\big(z,\cc{z};e^{it}\big)^{1-\mu}e^{-int}\, dt\\      &=(\mu)_nP^{1-\mu}_n(z,\cc{z})
    \overset{\eqref{eq:Poissonplusminus}}{=}(-\mu+1)_n P^{\mu}_n(z,\cc{z}) \, .
			\end{align*}
        The computation for $D^{0,n}$ and $P_{-n}^{\mu}$ is analogous.
	\end{proof}

  \begin{remark}
    Since Poisson Fourier modes are closely related to hypergeometric functions, see (\ref{eq:PoissonFourierHypergeometric1}), an alternative proof can be given by induction using Gauss' contiguous identities, more precisely with Eq.\ 15.2.24 and with a combination of Eq.\ 15.2.1 and 15.2.6 in \cite{abramowitz1984}.
  \end{remark}
		In view of \cite[Corollary~4.4]{HeinsMouchaRoth2}, we find another connection between the PFM and the pure Peschl--Minda derivatives. Let $\lambda=4m(m+1)$, $m,n\in\N_0$. It is true that
		\[P_n^{-m}\in\ker(D_w^{k})\qquad\text{and}\qquad P_{-n}^{-m}\in\ker(D_z^{k})\qquad\text{for all }k>m.\]
		\begin{proposition}\label{cor:PMkernels}
			Let $M\in\N_0$. Then
			\begin{equation}\label{eq:PMkernels}
				\{P_n^{-m}\, :\, m\in\N_0,\,n\in\Z,\, m\leq M,\, |n|\leq m\}=\ker(D^{M+1}_z)\cap\ker(D^{M+1}_w) \, .
			\end{equation}
		\end{proposition}
		\begin{proof}
        Denote the left--hand side of \eqref{eq:PMkernels} by $X_M$. Note that $\ker(D^{M+1}_z)\cap\ker(D^{M+1}_w)$ and $X_M$ are finite dimensional spaces. Moreover, $X_M$ is the direct sum
		\[X_M=\bigoplus_{m=0}^M X_{4m(m+1)}^0(\D^2)\]
		and
        $\dim X_{4m(m+1)}^0(\D^2)=2m+1$ by Theorem~\ref{thm:RudinthmforOmega}. Therefore, we compute on the one hand $\dim X_M = \sum_{m=0}^M (2m+1) = (M+1)^2$. On the other hand, by \cite[Corollary~4.4]{HeinsMouchaRoth2}, we have
            \begin{equation}
                \label{eq:PMKernelsExplicit}
                \ker(D^{M+1}_z)
                \cap
                \ker(D^{M+1}_w)
                =
                \mathrm{span}
                \bigg\{
                    \frac{z^j w^k}{(1-zw)^M}
                    \colon
                    0 \le j,k \le M
                \bigg\} \, .
            \end{equation}
            Consequently, $\dim(\ker(D^{M+1}_z)\cap\ker(D^{M+1}_w)) = (M+1)^2$. Moreover, combining \eqref{eq:PoissonFourierExplicit} with \eqref{eq:PMKernelsExplicit} yields $X_M\subseteq(\ker(D^M_z)\cap\ker(D^M_w))$. This proves \eqref{eq:PMkernels} by linear algebra.
		\end{proof}
		\begin{remark}
        Note that both the PFM $P_n^{-m}$ and the generators from \eqref{eq:PMKernelsExplicit} may be understood as holomorphic functions on all of $\Omega$. It thus makes sense to speak of
        \begin{equation*}
            \bigoplus_{M\in\N_0}\big(\ker(D_{z}^{M+1})\cap\ker(D_w^{M+1})\big)
        \end{equation*}
		as the subspace of $\mathcal{H}(\Omega)$ consisting of all finite linear combinations of elements in the kernels of the Peschl--Minda differential operators. We can even go one step further than Proposition~\ref{cor:PMkernels} and conclude that as sets
\begin{align*}
\clos{\Omega}\left( \mathrm{span}\{P_n^{-m}\,:\, m\in\N_0,\,n\in\Z,\, |n|\leq m\}\right)
  & =\clos{\Omega}\left(\bigoplus_{M\in\N_0}\big(\ker(D_{z}^{M+1})\cap\ker(D_w^{M+1})\big)\right)
  \\ & =\mathcal{H}(\Omega) \, ,\end{align*}
where the last equality follows from a combination of \eqref{eq:PMKernelsExplicit} with \cite[(4.1) and Cor.~4.8]{HeinsMouchaRoth1}. This observation leads to the question whether the PFM form a Schauder basis of $\mathcal{H}(\Omega)$. The second author answers this question affirmatively in \cite{Annika}. Perhaps similar results hold for $\mathcal{H}(\Omega_\pm)$ and the corresponding Poisson Fourier modes from Theorem \ref{thm:helgason}.
		\end{remark}


\begin{thebibliography}{99}

			\bibitem{abramowitz1984}
			M.~Abramowitz, I.A.~Stegun, M.~Danos, J.~Rafelski, \emph{Pocketbook of mathematical functions: Abridged edition of handbook of mathematical functions}, Verlag Harri Deutsch, (1984)

			\bibitem{Ahar69}
			D.~Aharonov, \emph{A necessary and sufficient condition for univalence of a
				meromorphic function}, Duke Math. J. \textbf{36} (1969), 599--604.

            \bibitem{askey} G.E..~Andrews, R.~Askey, R.~Roy, \emph{Special functions}, Cambr.~Univ.~Press (1999).

			\bibitem{BauerI} K.W.~Bauer, \emph{{\"u}ber die L{\"o}sungen der elliptischen Differentialgleichung $(1\pm z\bar{z})^2 w_{z\bar{z}}+\lambda w=0$ I},  J. Reine Angew. Math.~{\bf 221} (1966), 48–84.

			\bibitem{BauerII} K.W.~Bauer, \emph{{\"u}ber die L{\"o}sungen der elliptischen Differentialgleichung $(1\pm z\bar{z})^2 w_{z\bar{z}}+\lambda w=0$ II}, J. Reine Angew. Math.~{\bf 221} (1966),  176--196.

\bibitem{BauerRuscheweyh}
K.W.~Bauer and S.~Ruscheweyh,
\textit{Differential operators for partial differential equations and function theoretic applications},
Lecture Notes in Mathematics. 791. Berlin--Heidelberg--New York: Springer (1980).

	\bibitem{BeiserWaldmann2014}
				S.~Beiser and S.~Waldmann, \emph{Fr\'echet algebraic deformation quantization of the Poincar\'e disk}, J.~Reine Angew. Math.~688 (2014) 147--207.

			\bibitem{BerensteinGay1995} C.A.~Berenstein and R.~Gay, \emph{Complex Analysis and Special Topics in Harmonic Analysis}, Springer (1995).

			\bibitem{BroeckerDieck2003} T.~Br{{\"o}}cker and T.~Dieck, \emph{Representations of Compact Lie Groups}, Springer (2003)

   \bibitem{CahenGuttRawnsley1994}
				M.~Cahen, S.~Gutt and J.~Rawnsley,
				\emph{Quantization of K{\"{a}}hler Manifolds. III},
				Lett. Math. Phys., (1994), \textbf{30}, 291--305.

			\bibitem{Erdelyi1953} A.~Erdelyi, \emph{Higher transcendental functions}, McGraw--Hill (1953)

			\bibitem{EspositoSchmittWaldmann2019}
			C. Esposito, P. Schmitt and S. Waldmann, \emph{Comparison and continuity of Wick--type star products on certain coadjoint orbits}, Forum Math.31(5) (2019) 1203–1223.

            \bibitem{FischerLieb2012} W.~Fischer, I.~Lieb, \emph{A Course in Complex Analysis -- From Basic Results to Advanced Topics}, Vieweg \& Teubner (2012).

            \bibitem{FritzscheGrauert}
            K.~Fritzsche, H.~Grauert, \emph{From Holomorphic Functions to Complex Manifolds}, Springer (2002).

			\bibitem{gallier2020}
			J.~Gallier, J.~Quaintance, \emph{Spherical Harmonics and Linear Representations of Lie Groups}. In: Differential Geometry and Lie Groups. Geometry and Computing, vol 13. Springer, Cham. (2020)

			\bibitem{Hall2010}
			B.~Hall, \emph{Lie Groups, Lie Algebras, and Representations: An Elementary Introduction}, Springer (2010)

   			\bibitem{Harmelin1982}
			R.~Harmelin, \emph{Aharonov invariants and univalent functions}, Israel J. Math. {\bf 43} (1982), 244--254.

			\bibitem{HeinsMouchaRoth1} M.~Heins, A.~Moucha and O.~Roth, \textit{Function Theory off the complexified unit circle: Fr\'{e}chet space structure and automorphisms}, (2023), see \href{https://arxiv.org/abs/2308.01107}{arXiv:2308.01107}.

			\bibitem{HeinsMouchaRoth2} M.~Heins, A.~Moucha, O.~Roth and T.~Sugawa, \textit{Peschl--Minda derivatives and convergent Wick star products on the disk, the sphere and beyond}, (2023), see \href{https://arxiv.org/abs/2308.01101}{arXiv:2308.01101}.

            \bibitem{Helgason1959}
            S.~Helgason, \emph{Differential Operators on Homogeneous Spaces}, (1959), Acta Math. \textbf{102}, 239--299.

			\bibitem{Helgason70}
			S.~Helgason, \emph{A Duality for Symmetric Spaces with Applications to Group Representations}, (1970), Advances in Mathematics \textbf{5}, 1--154.

            \bibitem{Helgason84}
            S.~Helgason, \emph{Groups and Geometric Analysis: Integral Geometry, Invariant Differential Operators, and Spherical Functions}, (1984), Academic Press.

            \bibitem{Kashiwara}
            M.~Kashiwara, A.~Kowata, K.~Minemura, K.~Okamoto, T.~Oshima and M.~Tanaka, \emph{Eigenfunctions of invariant differential operators on a symmetric space}, Ann. Math., \textbf{107} (1978), 1--39.

			\bibitem{KS07diff}
			S.-A.~Kim and T.~Sugawa, \emph{Invariant differential operators associated with a
				conformal metric}, Michigan Math. J. \textbf{55} (2007), 459--479.

			\bibitem{KS11} S.-A.Kim and T.~Sugawa, \emph{Geometric invariants associated with projective structures and univalence criteria},  Tohoku Math. J. (2)
            {\bf 63} (2011), no.~1, 41--57.

            \bibitem{Koryani}
            A.~Koryani, \emph{A survey of harmonic functions on symmetric spaces}, Proc. Symp. Pure Math., \textbf{35}, Amer. Math. Soc., Providence, (1979), 323--340.

            \bibitem{KrausRothSchleissinger} D.~Kraus, O.~Roth, S.~Schlei{\ss}inger and S.~Waldmann, \textit{Strict Wick--type deformation quantization on Riemann surfaces: Rigidity and Obstructions}, (2023), see \href{https://arxiv.org/abs/2308.01114}{arXiv:2308.01114}.

			\bibitem{KrausRothSchoetzWaldmann}
			D.~Kraus, O.~Roth, M.~Sch\"{o}tz, S.~Waldmann , \emph{A Convergent Star Product on the Poincaré Disc},
			J. Funct. Anal., 277 no. 8, 2734--2771, (2019),
			see \href{https://arxiv.org/abs/1803.02763}{arXiv:1803.0276}.

            \bibitem{Maass1949}
            H.~Maaß, \emph{{\"u}ber eine neue Art von nichtanalytischen automorphen Funktionen und die Bestimmung Dirichletscher Reihen durch Funktionalgleichung}, Math. Ann., \textbf{121} (1949), 141--183.

            \bibitem{MacRobert} T.M.~MacRobert,  \textit{On an asymptotic expansion of the hypergeometric function}, Proc. Edinburgh Math. Soc. (1) 42 , 84--92, (1924)

			\bibitem{Minda} D.~Minda, unpublished notes.

			\bibitem{Annika} A.~Moucha, \textit{Spectral synthesis of the invariant Laplacian and complexified spherical harmonics}, (2023), see \href{https://arxiv.org/abs/2312.12931}{arXiv:2312.12931}.

            \bibitem{dlmf} NIST Digital Library of Mathematical Functions. https://dlmf.nist.gov/, Release 1.1.10 of 2023-06-15. F. W. J. Olver, A. B. Olde Daalhuis, D. W. Lozier, B. I. Schneider, R. F. Boisvert, C. W. Clark, B. R. Miller, B. V. Saunders, H. S. Cohl, and M. A. McClain, eds.

			\bibitem{Peschl1955}
   E. Peschl, \textit{Les invariants diff\'erentiels non holomorphes et leur r\^{o}le dans la th\'eorie des fonctions}, Rend. Sem. Mat. Messina Ser. I (1955), 100–108.

			\bibitem{Range}
   R.M.~Range, \emph{Holomorphic functions and integral representations in several complex variables}, Springer. (1986).

            \bibitem{Rudin83}
            W.~Rudin, \emph{Moebius--invariant algebras in balls},
            Annales Inst.~Fourier \textbf{33} no 2 (1983), p. 19--41.

			\bibitem{Rudin84}
			W.~Rudin, \emph{Eigenfunctions of the invariant Laplacian in $B$}, J.~d'Anal.~Math. \textbf{43} (1984), 136--148.

			\bibitem{Rudin2008}
			W.~Rudin, \emph{Function Theory in the unit ball of $\C^n$}, Springer, (2008)

	    \bibitem{Schippers2003}
				E.~Schippers, \emph{Conformal invariants and higher-order schwarz lemmas},
				J. Anal. Math. \textbf{90}, 217–241 (2003).

			\bibitem{Schippers2007}
				E.~Schippers, \emph{The calculus of conformal metrics},				Ann. Acad. Sci. Fenn., Math.~{\bf 32}, No.~2, 497--521 (2007).

   \bibitem{SchmittSchoetz2022}
		P.~Schmitt, M.~Sch{\"o}tz, \emph{Wick Rotations in Deformation Quantization}, Reviews in Mathematical Physics, Vol. 34, No. 1 (2022) 2150035.

		\bibitem{Sugawa2000}
			T. Sugawa, \emph{Unified approach to conformally invariant metrics on Riemann surfaces},			Proceedings of the Second ISAAC Congress, vol. 2 (Fukuoka, 1999), Int. Soc. Anal.			Appl. Comput., 8, pp. 1117–1127, Kluwer, Dordrecht, (2000).

			\bibitem{Waldmann2019}
				S.~Waldmann, \emph{Convergence of Star Products: From Examples to a General Framework}, EMS Surv.~Math.~Sci. 6 (2019), 1--31.

		\end{thebibliography}
	\end{document}